\documentclass{amsart}
\usepackage{sabbah_three-results}

\begin{document}
\frontmatter

\title{Three results on holonomic $\mathcal D$-modules}

\author[C.~Sabbah]{Claude Sabbah}
\address{Centre de Mathématiques Laurent Schwartz (CMLS), CNRS, École polytechnique, Institut Polytechnique de Paris, Palaiseau, France}
\email{Claude.Sabbah@polytechnique.edu}
\urladdr{https://perso.pages.math.cnrs.fr/users/claude.sabbah}

\begin{abstract}
In this text, we illustrate the use of local methods in the theory of (irregular) holonomic $\mathcal D$-modules.

I. (The Euler characteristic of the de~Rham complex) We show the invariance of the global or local Euler characteristic of the de~Rham complex after localization and co-localization of a holonomic $\mathcal D$-module along a hypersurface, as~well as after tensoring with a rank one meromorphic connection with regular singularities.

II. (Local generic vanishing theorems for holonomic $\mathcal D$-modules) We prove that the natural morphism from the proper pushforward to the total pushforward of an algebraic holonomic $\mathcal{D}$-module by an open inclusion is an isomorphism if we first twist the $\mathcal{D}$-module structure by suitable closed algebraic differential forms.

III. (Laplace transform of a Stokes-filtered constructible sheaf of exponential type) Motivated by the construction in \cite{Y-Z24}, we~propose a slightly different construction of the Laplace transform of a Stokes-perverse sheaf on the projective line and show directly that it corresponds to the Laplace transform of the corresponding holonomic D\nobreakdash-module via the Riemann-Hilbert-Birkhoff-Deligne-Malgrange correspondence. This completes the presentation given in \cite[Chap.\,7]{Bibi10}, where only the other direction of the Laplace transformation is analyzed. We~also compare our approach with the construction made previously in \cite{Y-Z24}.
\end{abstract}



\maketitle
\vspace*{-2\baselineskip}
{\let\nobreakspace\xspace\tableofcontents}
\mainmatter

\section*{Introduction}

The theory of holonomic \DXmodules (with possibly irregular singularities) on a complex manifold $X$ has been developed thanks to the fundamental result of K.\,Kedlaya \cite{Kedlaya10} and its algebraic analogue due to T.\,Mochizuki \cite{Mochizuki09,Mochizuki08}. The latter work makes it possible to handle Stokes structures in dimension~$\geq2$, as~already illustrated in several works, such as, in particular, \cite{Mochizuki10, Mochizuki10b,Bibi16b, Teyssier17}, and of course by the construction of the Riemann-Hilbert correspondence \cite{D-K13}.

In this note, we give a few more examples of applications of the methods developed in these articles in relation to questions raised or resolved in Gabriel Ribeiro's thesis \cite{Ribeiro24} and the article \cite{Y-Z24} by Tony Yue Yu and Shaowu Zhang. Our aim is to illustrate the use (or show how to avoid it) of Stokes structures in dimension $\geq2$. The main results are Theorems \ref{th:Euler} and \ref{th:Eulerlocal} in Part~\ref{part:Euler}, Theorem \ref{th:vanishingtwist} in Part \ref{part:vanishing}, and Theorem \ref{th:main} in Part \ref{part:Laplace}.

\subsubsection*{Notation}
We use the notation $\D f_*$ and $\D f_!$ for the pushforward and proper pushforward by a morphism $f$ in the bounded derived category of $\cD$-modules; they are also often denoted by $f_+$ and $f_\dag$, respectively. Similarly, the pullback of left $\cD$-modules, or~$\cO$-modules with flat connection, corresponding to $\bL f^*$ at the level of $\cO$-modules is denoted by $\D f^*$; it is also often denoted by $f^+$.

\subsubsection*{Acknowledgements}
I thank Gabriel Ribeiro for noticing a wrong statement in a previous version of Theorem \ref{th:vanishingtwist}, which has been deleted from this version. I thank Shaowu Zhang for interesting discussions about his joint paper \cite{Y-Z24} which led to the third part of this article.

\part{The Euler characteristic of the\nobreakspace de\nobreakspace Rham\nobreakspace complex}\label{part:Euler}

\section{Introduction to Part \ref{part:Euler}}
Let $X$ be a compact complex manifold of dimension $n$ and let $j:U\hto X$ be the inclusion of the complement of a reduced divisor $D$ in $X$, with complementary inclusion $i:D\hto X$. Let $\cD_X$ denote the sheaf of holomorphic differential operators on $X$ and let $\cD_X(*D)$ denote the sheaf of meromorphic differential operators with poles along $D$, that is, $\cD_X(*D)=\cO_X(*D)\otimes_{\cO_X}\cD_X=\cD_X\otimes_{\cO_X}\cO_X(*D)$. For a holonomic (left) \DXmodule $\cM$ (more generally a bounded complex of \DXmodules with holonomic cohomology, \ie an object of $\catD^\rb_\hol(\cD_X)$), we consider its de~Rham complex $\pDR(\cM)$ (with $\Omega^n_X\otimes\cM$ in degree zero for a holonomic \DXmodule), which is known to be a complex with $\CC$-constructible cohomology (and a perverse sheaf if~$\cM$ is a holonomic \DXmodule).

Let $\chi_\dR(X,\cM)$ denote the global Euler-Poincaré characteristic of the de~Rham complex $\pDR(\cM)$:
\[
\chi_\dR(X,\cM):=\sum_k(-1)^k\dim\bH^k(X,\pDR(\cM)).
\]
To $\cM$ are associated its localized module (or complex) $\cM(*D)=\cD_X(*D)\otimes_{\cD_X}\cM$ along~$D$, which is known to be holonomic (or with holonomic cohomology), and its dual-localized module (or complex) $\cM(!D)$ defined by duality from $\cM(*D)$. Letting~$\bD$ denote the duality functor for coherent \DXmodules, we know that $\bD$ preserves the category of (bounded complexes of) holonomic \DXmodules. We~then set
\[
\cM(!D):=\bD\bigl[(\bD\cM)(*D)\bigr].
\]
There exists a natural morphism $\cM(!D)\to\cM(*D)$ whose cone has cohomology supported on $D$.

By a meromorphic flat bundle on $(X,D)$, we mean a holonomic \DXmodule~$\cM$ which is a vector bundle with integrable connection on $U=X\moins D$ and satisfies $\cM=\cM(*D)$. The dual $\cM^\vee$ of such a meromorphic flat bundle $\cM$ is the meromorphic flat bundle $\cM^\vee:=(\bD\cM)(*D)$ with underlying $\cO_X(*D)$-module $\cHom_{\cO_X(*D)}(\cM,\cO_X(*D))$. Note that there are no $\cExt$, as expected from the definition of $\cM^\vee$, since $\cM$ is locally stably $\cO_X(*D)$-free, by \cite[Prop.\,1.1]{Malgrange95}.

\begin{theoreme}\label{th:Euler}
Let $\cM$ be an object of $\catD^\rb_\hol(\cD_X)$ and let $\cN$ be a meromorphic flat bundle of rank one on $X$ with poles along~$D$ at most and having regular singularities along~$D$. Then the global Euler-Poincaré characteristics of the de~Rham complexes~of
\[
\cM(!D),\quad\cM(*D),\quand \cM\otimes \cN
\]
coincide.
\end{theoreme}

\begin{remarque}\mbox{}
\begin{enumerate}
\item(Algebraic version)
If $U$ is smooth and quasi-projective, a similar result holds in the framework of algebraic \DUmodules, and is obtained from Theorem \ref{th:Euler} by choosing any smooth projectivization $X$ of $U$ and by making use of GAGA for algebraic \DXmodules. For a coherent algebraic $\cD_U$-module $M$, the proper and total pushforwards $\D j_!M$ and $\D j_*M$ are defined and are holonomic if $M$ is so.

It states, for an algebraic $\cD_U$-module $M$ and an algebraic vector bundle~$N$ of rank one on $U$ with a flat connection having a regular singularity at infinity, that
\begin{starequation}\label{eq:chialg}
\chi_{\dR,\rc}(U,M)=\chi_{\dR}(U,M)=\chi_{\dR}(U,M\otimes N),
\end{starequation}%
with (due to a GAGA type argument for the second equalities)
\begin{align*}
\chi_{\dR,c}(U,M)&:=\sum_k(-1)^k\dim\bH^k(X,\pDR(\Dj_!M))\\
&\hphantom{:}=\sum_k(-1)^k\dim\bH^k(X^\an,\pDR^\an(\Dj_!M)),\\
\chi_\dR(U,M)&:=\sum_k(-1)^k\dim\bH^k(X,\pDR(\Dj_*M))\\
&\hphantom{:}=\sum_k(-1)^k\dim\bH^k(X^\an,\pDR^\an(\Dj_*M)).
\end{align*}
(If $N$ has rank $r$, then $\chi_{\dR}(U,M\otimes N)=r\cdot\chi_{\dR}(U,M)$.)
\item
The first equality of \eqref{eq:chialg} is proved by Gabriel Ribeiro in the appendix of his PhD thesis \cite{Ribeiro24}, and the second equality is mentioned as a conjecture, with a nice algebraic proof only in the case where $D=\emptyset$. In~this part, we~show how local methods allow us to prove the theorem in the analytic setting, so that the theorem in the algebraic setting follows by a GAGA type argument. The proof is of a transcendental nature and relies on the notion of the irregularity complex, due to Z.\,Mebkhout \cite{Mebkhout90,Mebkhout04}, and on the theorem of Kedlaya and Mochizuki \cite{Kedlaya09,Kedlaya10,Mochizuki07b,Mochizuki08} on the normal form of meromorphic flat bundles after blowing up.
\end{enumerate}
\end{remarque}

We will thus focus on a local analytic version of the theorem. For a holonomic \DXmodule~$\cM$ and a point $x\in X$, the cohomology groups of the complex of germs $\pDR(\cM)_x$ are finite-dimensional vector spaces, and the local Euler characteristic of the de~Rham complex at~$x$ is the integer
\[
\chi_{\dR,x}(\cM)=\sum_k(-1)^k\dim\cH^k(\pDR(\cM)_x)=:\chi_x(\pDR(\cM)).
\]

\begin{theoreme}\label{th:Eulerlocal}
Let $\cM$ be an object of $\catD^\rb_\hol(\cD_X)$ and let $\cN$ be a meromorphic flat bundle of rank one on $X$ with poles along~$D$ at most and having regular singularities along~$D$. Then, for each $x\in X$, the local Euler-Poincaré characteristics at $x$ of the de~Rham complexes of the following objects of $\catD^\rb_\hol(\cD_X)$ coincide:
\[
\cM(!D),\quad\cM(*D),\quand \cM\otimes \cN.
\]
\end{theoreme}

\begin{proof}[Proof that Theorem \ref{th:Eulerlocal} implies Theorem \ref{th:Euler}]
For any complex~$\cF$ with $\CC$\nobreakdash-construc\-tible cohomology on $X$, whose cohomology is locally constant on the strata of a (Whitney) stratification $(X_\alpha)_\alpha$ of $X$ by locally closed smooth subvarieties, the Euler characteristic
\[
\chi(\cF_x)=\sum_k\dim\cH^k(\cF_x)
\]
does not depend on the choice of $x$ in $X_\alpha$ and is denoted by $\chi_\alpha(\cF)$. We~then have
\begin{equation}\label{eq:chiXF}
\chi(X,\cF)=\sum_\alpha \chi_\alpha(\cF)\cdot\chi(X_\alpha).
\end{equation}
In particular, for two such complexes $\cF_1,\cF_2$, we have $\chi(X,\cF_1)=\chi(X,\cF_2)$ as soon as $\chi(\cF_{1,x})=\chi(\cF_{2,x})$ for all $x$ in $X$: indeed, we can find a stratification~$(X_\alpha)_\alpha$ suitable for both, and the condition is equivalent to $\chi_\alpha(\cF_1)=\chi_\alpha(\cF_2)$ for all~$\alpha$; we~then apply Formula \eqref{eq:chiXF}.
\end{proof}

\begin{remarque}
For a meromorphic flat bundle on $(X,D)$, Theorem \ref{th:Eulerlocal} amounts to the equality, for any $x\in X$,
\[
\chi_{\dR,x}(\cM)=\chi_{\dR,x}(\cM^\vee)=\chi_{\dR,x}(\cM\otimes\cN).
\]
In other words, we claim that $\chi_{\dR,x}(\cM(!D))=\chi_{\dR,x}(\cM^\vee)$. This equality holds because
\begin{itemize}
\item
we have by definition $\cM(!D)=\bD(\cM^\vee)$;
\item
if $\bD_\PV$ denotes the Poincaré-Verdier duality on the category of bounded complexes with constructible cohomology on $X^\an$, then we have
\[
\chi_x\bigl(\bD_{\PV}(\cF)\bigr)=\chi_x\bigl(\cF\bigr);
\]
\item
for a holonomic bounded complex $\cM$, we have $\pDR(\bD\cM)\simeq\bD_\PV(\pDR(\cM))$.
\end{itemize}
\end{remarque}

We will show in Section \ref{sec:reduction} how to reduce the proof of Theorem \ref{th:Eulerlocal} to that of the following theorem. The notions of good formal structure on a meromorphic flat bundle and of a good meromorphic flat bundle are recalled in Section \ref{subsubsec:goodformalstr}.

\begin{theoreme}\label{th:Eulerlocalgood}
Assume that $D$ is a divisor with simple normal crossings in $X$ and that $\cM=\cM(*D)$ is a good meromorphic flat bundle on $(X,D)$. Let $\cN$ be as in Theorem~\ref{th:Euler}. Then, for any $x\in X$, the following local Euler characteristics coincide:
\begin{starequation}\label{eq:Eulerlocalgoodstar}
\chi_{\dR,x}(\cM)=\chi_{\dR,x}(\cM^\vee)=\chi_{\dR,x}(\cM\otimes\cN).
\end{starequation}%
\end{theoreme}

Let us recall that the \emph{irregularity complex} $\Irr_D(\cM)$ of $\cM$ along~$D$, as introduced by Mebkhout\footnote{In \cite{Mebkhout90,Mebkhout04}, the complex $\Irr_D(\cM)$ is defined in terms of the unshifted de~Rham complex $\DR(\cM)$, while here we use the shifted one; nevertheless, the same shift, with respect to \eqref{eq:Eulerlocalgoodstar}, occurs in the terms of \eqref{eq:Eulerlocalgood}.} \cite{Mebkhout90,Mebkhout04}, is defined as the kernel, in the perverse category, of the natural morphism $\pDR(\cM(*D))\to\bR j_*j^{-1}\pDR\cM$.

\begin{definition}[Irregularity number]
The \emph{irregularity number of $\cM$ at $x\in D$} is defined as
\[
\irr_x(\cM,D):=\chi(\Irr_D(\cM)_x).
\]
\end{definition}

Given a locally constant sheaf $\cV$ on $U=X\moins D$, with $D$ being a normal crossing divisor, the Euler characteristic number $\chi((\bR j_*\cV)_x)$ vanishes at every point $x\in\nobreak D$ because there exists a fundamental system of open neighborhoods of $x$ whose trace on $X\moins D$ has the homotopy type of a finite product of circles. Therefore, the statement of Theorem \ref{th:Eulerlocalgood} amounts to the equality of irregularity numbers
\begin{equation}\label{eq:Eulerlocalgood}
\irr_x(\cM,D)=\irr_x(\cM^\vee,D)=\irr_x(\cM\otimes\cN,D)\quad\forall x\in D.
\end{equation}

\section{Proof of Theorem \ref{th:Eulerlocalgood}}\label{sec:Eulerlocalgood}

Under the assumptions of Theorem \ref{th:Eulerlocalgood}, we will prove the equalities \eqref{eq:Eulerlocalgood}. We~first recall some parts of Section~2 of \cite{Bibi16b}.

\subsubsection{Setting}\label{subsec:nota}
Let $X$ be a complex manifold, let $D$ be a divisor with simple normal crossings in $X$, \ie we assume that each irreducible component~$D_i$ ($i\in J$) of $D$ is smooth, and set $U=X\moins D$. For any subset $I\subset J$, we~set $D_I=\bigcap_{i\in I}D_i$ and $D_I^\circ=D_I\moins\bigcup_{j\notin I}D_j$. We~denote the codimension of $D_I^\circ$ by $\ell$, which we regard as a locally constant function on $D_I^\circ$ (which can have many connected components), and by $\iota_I:D_I^\circ\hto D$ the inclusion.

If $Z$ is any locally closed analytic subset of $X$, we~denote by $\cO_{\wh Z}$ the formal completion of $\cO_X$ with respect to the ideal sheaf $\cI_Z$. We~regard $\cO_{\wh Z}$ as a sheaf on~$Z$.

Given $x_o\in D$, there exists a unique subset $I\subset J$ such that $x_o\in D_I^\circ$, and we will be mostly interested in the cases where $Z$ is either the point $x_o\in D$ or the locally closed subset $D_I^\circ$. We~will denote by $\cO_{\wh Z}(*D)$ the sheaf $\cO_{X|Z}(*D)\otimes_{\cO_{X|Z}}\cO_{\wh Z}$, where as usual $\cO_{X|Z}$ (\resp $\cO_{X|Z}(*D)$) denotes the sheaf-theoretic restriction to $Z$ of the sheaf~$\cO_X$ (\resp the sheaf $\cO_X(*D)$).

If $\varphi$ (\resp $\wh\varphi$) is a section of $\cO_X(*D)$ (\resp of $\cO_{\wh Z}(*D)$), we~denote by $\cE^\varphi$ (\resp $\cE^{\wh\varphi}$) the module with connection $(\cO_X(*D),\rd+\rd\varphi)$ (\resp $(\cO_{\wh Z}(*D),\rd+\rd\wh\varphi)$). It~only depends on the class, also denoted by $\varphi$ (\resp $\wh\varphi$), of $\varphi$ (\resp $\wh\varphi$) modulo~$\cO_X$ (\resp $\cO_{\wh Z}$).

For every $I\subset J$, we~consider the sheaf $\cO_{\wh{D_I^\circ}}$ on $D_I^\circ$. We~also regard it as a sheaf on $X$ by extending it by zero. We~then set $\cD_{\wh{D_I^\circ}}=\cO_{\wh{D_I^\circ}}\otimes_{\cO_X}\cD_X$, and $\cM_{\wh{D_I^\circ}}:=\cD_{\wh{D_I^\circ}}\otimes_{\cD_X}\cM$.

\subsubsection{A reminder of the notion of good formal structure}\label{subsubsec:goodformalstr}
Let $\cM$ be a meromorphic flat bundle on $(X,D)$. We~say that $\cM$ is \emph{good}, or~has a \emph{good formal structure} along~$D$, if, for any $x_o\in D$, there exists a local ramification $\rho_{\bmd_I}$ of multi-degree $\bmd_I$ around the branches $(D_i)_{i\in I}$ passing through $x_o$ (hence inducing an isomorphism above~$D_I^\circ$ in the neighborhood of $x_o$) such that the pullback of the formal meromorphic flat bundle $\cM_{\wh{x_o}}:=\cO_{\wh x_o}\otimes_{\cO_{X,x_o}}\cM_{x_o}$ by this ramification decomposes as the direct sum of formal elementary meromorphic flat bundles $\cE^{\wh\varphi}\otimes\wh\cR_{\wh\varphi}$, as~defined below.

We denote by $\nb(x_o)$ a small analytic open neighborhood of $x_o$ in $X$ above which the ramification is defined, and we denote by $x'_o$ the pre-image of $x_o$, so the ramification is a finite morphism $\rho_{\bmd_I}:\nb(x'_o)\to\nb(x_o)$. It~induces a one-to-one map above $D_I^\circ\cap\nb(x_o)$. We~also set $D'=\rho_{\bmd_I}^{-1}(D\cap\nb(x_o))$, so that $D'_I$ maps isomorphically to $D_I\cap\nb(x_o)=D_I^\circ\cap\nb(x_o)$.

In the above decomposition, $\wh\varphi$ varies in a \emph{good} finite subset $\wh\Phi_{x_o}\subset\cO_{\wh{x'_o}}(*D')/\cO_{\wh{x'_o}}$ and~$\wh\cR_{\wh\varphi}$ is a free $\cO_{\wh{x'_o}}(*D')$-module with an integrable connection having a regular singularity along $D'$. In~other words, we~do not distinguish between $\wh\varphi$ and $\wh\psi$ in $\cO_{\wh{x'_o}}(*D')$ if their difference has no poles along $D'$. Goodness means here that for any pair \hbox{$\wh\varphi\neq\wh\psi\in\wh\Phi_{x_o}\cup\{0\}$}, the difference $\wh\varphi-\wh\psi$ can be written as $x^{-\bme}\wh u(x)$, with $\bme\in\NN^{\#I}$ and $\wh u\in\cO_{\wh{x'_o}}$ satisfying $\wh u(0)\neq0$ (\cf \cite[\S I.2.1]{Bibi97}). By~\cite[Prop.\,4.4.1\,\&\,Def.\,5.1.1]{Kedlaya10} (\cf also \cite[\S I.2.4]{Bibi97} and \cite[Prop.\,2.19]{Mochizuki10b}), the~$\wh\varphi$ are convergent, \ie the set~$\wh\Phi_{x_o}$ is the formalization at $x_o$ of a finite subset $\Phi_{x_o}\subset\nobreak\Gamma(\nb(x'_o),\cO_{\nb(x'_o)}(*D')/\cO_{\nb(x'_o)})$, and the decomposition extends to a neighborhood of $x'_o$, that is, it~holds for the pullback by $\rho_{\bmd_I}$ of $\cM_{\wh{D_I^\circ},x_o}$ and induces the original one after taking formalization at $x'_o$.

It is also important to note the following openness property. The components $\cE^{\wh\varphi}\otimes\wh\cR_{\wh\varphi}$ are the formalization at $x'_o$ of local analytic meromorphic flat bundles $\cE^{\varphi}\otimes\cR_{\varphi}$ defined in some analytic neighborhood of $x'_o$. Let us decompose $D'$ into its irreducible components $D'_i$ with $i\in I$. Then, for any $x'\in D'$ in a sufficiently small analytic neighborhood $\nb(x'_o)$ of $x'_o$, the formalization $(\D\rho_{\bmd_I}^*\cM)_{\wh{x'}}$ decomposes with the components $(\cE^{\varphi}\otimes\cR_{\varphi})_{\wh{x'}}$ for $\varphi\in\Phi_{x_o}$.

\subsubsection{Omitting the Stokes filtration}\label{subsubsec:omitting}
Let us fix a subset $I\subset\{1,\dots,\ell\}$ and let us focus on proving \eqref{eq:Eulerlocalgood} at $x\in D_I^\circ$. One can express the complex $\Irr_D(\cM)$ in the neighborhood of~$x$ by means of the Stokes filtration of $\cM$ in the neighborhood of~$x$, and more precisely by means of the sheaf of rapid decay flat sections of $\cM$ on the real blow-up of $D$ along its irreducible components (\cf\cite[Cor.\,2.11]{Bibi93} and \cite[Prop.\,3.1]{Bibi16b}). This reduces the computation of the various Euler characteristics to a topological computation. However, the Stokes filtration can make a direct topological computation complicated. Fortunately, \cite[Th.\,1.2]{Bibi16b} allows us to act as though the Stokes filtration were trivial. More precisely, in the neighborhood of $D_I^\circ$, there exists a meromorphic flat bundle with trivial Stokes structure along~$D$, that is, the meromorphic flat bundle decomposes locally (after a possible local ramification) in a way similar to that of the formal decomposition of $\cM$, and its irregularity complex along $D_I^\circ$ is isomorphic to that of $\cM$. A fortiori, one can compute the various irregularity numbers of \eqref{eq:Eulerlocalgood} at $x\in D_I^\circ$ by means of such a meromorphic flat bundle, which we still denote by $(\cM,\nabla)$.

\subsubsection{Reduction to the case ``good of rank one''}\label{subsubsec:reductionrkone}
Let $\varpi$ denote the real blowup of the components of $D$, let $\scrL$ denote the local system on $\varpi^{-1}(x)\simeq (S^1)^\ell$ consisting of horizontal sections on $X\moins D$ of the above meromorphic flat bundle $(\cM,\nabla)$, and let \hbox{$\scrL_{\leq0}\!\subset\!\scrL$} be the subsheaf of those that have moderate growth in the neighborhood of~$D$. We~set $\scrL^{>0}=\scrL/\scrL_{\leq0}$. Then, by \cite[Prop.\,3.1]{Bibi16b}, we have $\Irr_D(\cM,\nabla)[1]=\bR\varpi_*\scrL^{>0}$, and therefore, as~$\chi((S^1)^\ell,\scrL)=0$, we find $\irr_x(\cM,\nabla)=\chi((S^1)^\ell,\scrL_{\leq0})$.

Let $\rho$ be a local ramification as in Section \ref{subsubsec:goodformalstr}, sending $x'$ to $x$. It~induces a covering $\rho:\varpi^{-1}(x')=(S^1)^\ell\to\varpi^{-1}(x)=(S^1)^\ell$, and the degree $\leq0$ of the Stokes filtration of $\rho^*(\cM,\nabla)$ is identified with $\rho^{-1}(\scrL_{\leq0})$. On the other hand, we have
\[
\chi\bigl((S^1)^\ell,\rho^{-1}(\scrL_{\leq0})\bigr)=\deg\rho\cdot\chi\bigl((S^1)^\ell,\scrL_{\leq0}\bigr),
\]
that is, $\irr_{x'}(\rho^*(\cM,\nabla))=\deg\rho\cdot \irr_x(\cM,\nabla)$. It~is thus enough to show the equalities~\eqref{eq:Eulerlocalgood} after ramification.

By~the previous reduction of Section \ref{subsubsec:omitting}, we can assume that $\rho^*\cM$ has a \emph{convergent} decomposition mimicking its good formal decomposition. In~such a case, $\rho^*\cM$ is a successive extension of good meromorphic flat bundles of rank one (because so is any meromorphic flat bundle with regular singularities along $D$). Since the irregularity number behaves additively in extensions, we can assume from the start that~$\cM$ is a good meromorphic flat bundle of rank one.

\subsubsection{Computation of the irregularity number of a good meromorphic flat bundle of rank one}\label{subsubsec:rankone}
Let us be given
\begin{itemize}
\item
a meromorphic function $\varphi(x)=u(x)/x^{\bme}$, with $u$ holomorphic and invertibl,e and $\bme=(e_1,\dots,e_\ell)\in\NN^\ell$ (if $\bme=0\in\NN^\ell$, we can choose $u\equiv1$),

\item
a vector $\alpha=(\alpha_1,\dots,\alpha_\ell)\in\CC^\ell$,
\end{itemize}
and we will compute the irregularity number at the origin of the germ of the trivial bundle of rank one
\[
(\cE,\nabla)=\biggl(\CC\{x_1,\dots,x_n\}[(x_1\cdots x_\ell)^{-1}],\rd+\rd\varphi+\sum_{i=1}^\ell\alpha_i\frac{\rd x_i}{x_i}\biggr).
\]
We note that, if $(\cN,\nabla)=\bigl(\CC\{x_1,\dots,x_n\}[(x_1\cdots x_\ell)^{-1}],\rd+\sum_{i=1}^\ell\beta_i\sfrac{\rd x_i}{x_i}\bigr)$,\vspace*{-3pt}
\begin{align*}
(\cE,\nabla)^\vee&=\biggl(\CC\{x_1,\dots,x_n\}[(x_1\cdots x_\ell)^{-1}],\rd-\rd\varphi-\sum_{i=1}^\ell\alpha_i\frac{\rd x_i}{x_i}\biggr),\\
(\cE,\nabla)\otimes(\cN,\nabla)&=\biggl(\CC\{x_1,\dots,x_n\}[(x_1\cdots x_\ell)^{-1}],\rd+\rd\varphi+\sum_{i=1}^\ell(\alpha_i+\beta_i)\frac{\rd x_i}{x_i}\biggr).
\end{align*}

The equalities \eqref{eq:Eulerlocalgood} for $\cM=(\cE,\nabla)$ are an immediate consequence of the first part of the next lemma.

\begin{lemme}
The irregularity number of $(\cE,\nabla)$ at the origin only depends on the family $(e_1,\dots,e_\ell)$ (with the convention that $\bme=0$ if $\varphi\equiv0$).
\end{lemme}

\begin{proof}
We keep the notation of Section \ref{subsubsec:reductionrkone}. The sheaf $\scrL_{\leq0}$ is described, and the irregularity number is then obtained, as follows:

\begin{enumeratei}
\item\label{enum:phinul}
if $\varphi\equiv0$, $\scrL_{\leq0}$ is the locally constant sheaf of rank $1$ and monodromy $\exp(-2\pi\sfi\alpha_i)$ in direction $i$; in this case, we have\vspace*{-3pt}\enlargethispage{.5\baselineskip}
\[
\irr_0(\cE,\nabla)=\chi((S^1)^\ell,\scrL_{\leq0})=\chi((S^1)^\ell)=0,
\]
whatever the $\alpha_i$ are.

\item\label{enum:phinonnul}
if $\varphi\not\equiv0$, $\scrL_{\leq0}$ is the extension by zero of the restriction of the above locally constant sheaf to the open subset $(S^1)^\ell_{\varphi,\bme}\subset(S^1)^\ell$ defined by the inequalities\vspace*{-3pt}
\[
\reel\bigl(\varphi(0)\exp(\textstyle\sum_i-e_i\theta_i)\bigr)>0,\quad i=1,\dots,\ell.
\]
\begin{itemize}
\item
If $\ell\geq2$, then $(S^1)^\ell_{\varphi,\bme}$ contains $S^1$ as a factor, hence $\chi_\rc\bigl((S^1)^\ell_{\varphi,\bme})=0$.

\item
If $\ell=1$, then $(S^1)_{\varphi,e_1}$ is the disjoint union of $e_1$ open intervals of $S^1$, and $\chi_\rc\bigl((S^1)_{\varphi,e_1}\bigr)=-e_1$.\qedhere
\end{itemize}
\end{enumeratei}
\end{proof}

\section{Reduction of the proof of Theorem \ref{th:Eulerlocal} to that of Theorem \ref{th:Eulerlocalgood}}\label{sec:reduction}

Although this reduction is very standard, we give a detailed proof for the sake of completeness. Given the additivity of $\chi_\dR$ and $\chi_{\dR,\rc}$ in exact sequences, we can assume that $\cM$ is a holonomic \DXmodule. Then we argue by induction on the dimension $d$ of the support $Z$ of $\cM$. Since\footnote{Use that $\cN(!D)=0$ if $\cN$ is supported on $D$, and that the kernel and cokernel of $\cM\to\cM(*D)$ are supported on $D$.} $\cM(!D)=[\cM(*D)](!D)$ and $\cM\otimes\cN=\cM(*D)\otimes\cN$, we can assume that $\cM=\cM(*D)$. The assertion is tautological if $\dim\Supp\cM=0$, and we assume it to be true if $\dim\Supp\cM<d$. As the assertion is local, we can assume that there exists a hypersurface $H$ in $X$ containing $D$ such that $Z^\circ:=Z\moins(Z\cap H)$ is smooth of pure dimension $d$, that no irreducible component of $Z$ of dimension $d$ is contained in $H$, and that the restriction of~$\cM$ to~$Z^\circ$ in the sense of $\cD$-modules (via Kashiwara's equivalence) is a holomorphic vector bundle over $Z^\circ$ with integrable connection. By~induction, it~is enough to prove the theorem for the localized module $\cM(*H)$.

Let $\pi:X'\to X$ be a projective morphism such that
\begin{itemize}
\item
$X'$ is smooth of pure dimension $d$,
\item
$H':=\pi^{-1}(H)$ is a divisor with normal crossings,
\item
$\pi:X^{\prime\circ}:=X'\moins H'\to Z^\circ$ is an isomorphism.
\end{itemize}
Such a morphism exists by resolving the singularities of $(Z_d,H\cap Z_d)$, where $Z_d$ is the union of the pure $d$-dimensional components of $Z$. The inverse image $\D\pi^*\cM(*H)$ of $\cM(*H)$ is thus a meromorphic flat bundle on $(X',H')$. Furthermore, by the theorem of Kedlaya-Mochizuki \cite[Th.\,8.1.3]{Kedlaya10}, we can moreover assume that $\D\pi^*\cM(*H)$ is a \emph{good} meromorphic flat bundle on $(X',H')$.

Set $D'=\pi^{-1}(D)$. Then $\cM(*H)$, $\cM(*H)(!D)$, and $\cM(*H)\otimes\cN$ can be recovered respectively from $\D\pi^*\cM(*H)$, $[\D\pi^*\cM(*H)](!D')$, and $[\D\pi^*\cM(*H)]\otimes\D\pi^*\cN$ as their pushforward $\D\pi_*(\D\pi^*\cM(*H))$, etc. As a consequence, $\pDR(\cM(*H))\simeq\bR\pi_*\pDR(\D\pi^*\cM(*H))$, etc., and the constructible function $x\!\mto\!\chi_{\dR,x}(\pDR(\cM(*H))$, etc., is equal to the pushforward, in the sense of \cite{MacPherson74}, of the constructible function $x'\mto\chi_{\dR,x'}\bigl(\pDR(\D\pi^*\cM(*H))\bigr)$, etc.

We conclude that it is enough to prove Theorem \ref{th:Eulerlocal} in the setting of Theorem~\ref{th:Eulerlocalgood}, and this has been achieved in Section \ref{sec:Eulerlocalgood}.\qed

\part{Local generic vanishing theorems for holonomic \texorpdfstring{$\mathcal D$}{D}-modules}\label{part:vanishing}

\section{Local and relative generic vanishing theorems}
In his PhD thesis \cite{Ribeiro24}, Gabriel Ribeiro suggests the statement of a ``relative vanishing theorem'', which we recall in Section \ref{subsubsec:appl}. We~give here a complete proof of a ``local generic vanishing theorem'' that implies the corresponding ``relative'' statement, and which we regard as a corollary of a more general result stated in Section~\ref{subsubsec:genres}: the natural morphism from the proper pushforward to the total pushforward of an algebraic holonomic $\cD$-module by an open inclusion is an isomorphism if we first twist the $\cD$-module structure by suitable closed algebraic differential forms.

\subsubsection{A general result}\label{subsubsec:genres}
Let $U$ be a smooth complex algebraic variety and let~$\cM_U$ be a quasi-coherent (left) $\cD_U$-module, which we regard as a quasi-coherent $\cO_U$-module with integrable connection $\nabla$. Furthermore, let $\etag^o=(\eta_1^o,\dots,\eta_r^o)$ be a finite family of regular closed one-forms on~$U$. For any $a\in\CC^r$, we~consider the twisted $\cD_U$-module
\[\textstyle
\cM_{U,\etag^o,a}:=\bigl(\cM_U,\nabla+\sum_ka_k\eta^o_k\bigr).
\]
This construction can be extended to any object $\cM_U\in\catD^\rb_\mathrm{qcoh}(\cD_U)$, and the $(\etag^o,a)$ twist commutes with taking cohomology. If $\cM_U$ belongs to $\catD^\rb_\hol(\cD_U)$, so does $\cM_{U,\etag^o,a}$. Let~$X$ be any smooth compactification of $U$, that is, a smooth proper complex algebraic variety that contains $U$ as an open dense subvariety, and let $j:U\hto X$ denote the open inclusion. We~say that $X$ is a \emph{good compactification of~$U$} if, moreover, $D:=X\moins U$ is a divisor with simple normal crossings in $X$ (\ie each irreducible component of $D$ is smooth). Good compactifications exist: starting from any Nagata compactification of $U$, one blows up above $X\moins U$ according to Hironaka's theorem on embedded resolution of singularities.

For any smooth compactification $X$ of $U$, we~denote by
\[
\Dj_*,\text{ \resp }\Dj_!:\catD^\rb_\hol(\cD_U)\to\catD^\rb_\hol(\cD_X)
\]
the direct image, \resp the direct image with proper support (\ie $\Dj_!=\bD\,\Dj_*\bD$, where~$\bD$ is the duality functor). For $\cM\in\catD^\rb_\hol(\cD_U)$, there exists a natural morphism (\cf \cite[Th.3.2.16]{H-T-T08})
\begin{equation}\label{eq:dag*a}
\Dj_!\cM_{U,\etag^o,a}\to \Dj_*\cM_{U,\etag^o,a},
\end{equation}
where both terms are objects of $\catD^\rb_\hol(\cD_X)$.

\begin{definition}[Local vanishing]\label{def:relvanish}
We~say that local generic vanishing holds for $(X,\cM_U,\etag^o)$ if there exists a dense open subset $V\subset\CC^r$ such that, for each $a\in V$, \eqref{eq:dag*a} is an isomorphism.
\end{definition}

Let us start by noting:

\begin{lemme}\label{lem:1}
Assume that a smooth compactification $X'$ of $U$ dominates $X$, \ie there exists a proper modification $\pi:X'\to X$ which is the identity on $U$. If local generic vanishing holds for $(X',\cM_U,\etag^o)$, then it holds for $(X,\cM_U,\etag^o)$.
\end{lemme}

\begin{proof}
Let $j':U\hto X'$ denote the open embedding. As $\pi$ is proper, we~have $\D\pi_!=\D\pi_*$, and the natural morphism $\Dj'_!(\cbbullet)\to \Dj'_*(\cbbullet)$ yields the natural morphism
\[
\bigl[\D\pi_!\, \Dj'_!(\cbbullet)\to \D\pi_*\,\Dj_*(\cbbullet)\bigr]\simeq \bigl[\Dj_!(\cbbullet)\to \Dj_*(\cbbullet)\bigr].
\]
If the first one is an isomorphism when applied to $\cM_{U,\etag^o,a}$, so is the second one.
\end{proof}

Let $(Y,D)$ be a \emph{good} compactification of $U$ and let $\eta$ be a closed logarithmic one-form on $Y$ with poles along $D=\bigcup_{i\in J}D_i$, with $D_i$ smooth and connected. Let~$D^o$ denote the smooth locus of $D$. Since $\eta$ is closed, the residue $\res_i(\eta)$ of $\eta$ along $D_i^o:=D_i\cap D^o$ (which is a holomorphic function on $D_i^o$) is constant. For any subset $I\subset J$ such that $\bigcap_{i\in I}D_i\neq\emptyset$, set $I(\eta)=\{i\in I\mid \res_i(\eta)\neq0\}$. Equivalently, $I(\eta)$ is the set of indices $i$ in $I$ such that $\eta$ has a pole along $D_i$.

\begin{definition}[Non-resonance]
We~say that $\eta$ is \emph{$I$-non-resonant} if $I(\eta)\neq\emptyset$ and there does not exist any relation $\sum_{i\in I(\eta)}m_i\res_i(\eta)=0$ with $(m_i)_{i\in I(\eta)}\in\NN^{\#I(\eta)}\moins\{0\}$.
\end{definition}

For example, if $\eta$ has a pole along only one irreducible component $D_i$ of $D$, it~is $I$-non-resonant for any subset~$I$ containing $i$.

\begin{theoreme}\label{th:vanishingtwist}
Let $\eta^o_1,\dots\eta^o_r$ be a family of closed one-forms on $U$. Assume that there exists a good compactification $(Y,D)$ of $U$ such that their rational extensions $\eta_1,\dots\eta_r$ are (closed) logarithmic one-forms on $Y$ with poles along $D$ satisfying
\begin{starequation}\label{eq:vanishingtwist*}
\begin{cases}
\text{for each non-empty intersection $\bigcap_{i\in I}D_i$, there exists $k\in\{1,\dots,r\}$}\\
\text{such that $\eta_k$ is $I$-non-resonant.}
\end{cases}
\end{starequation}%
Then, given any holonomic $\cD_U$-module $\cM_U$ (or any object in $\catD^\rb_\hol(\cD_U)$), for any compactification $X$ of $U$, there exists a finite set of affine forms
\[
L_{\bmell,c}(a)=\sum_k\ell_k a_k+c,
\]
with $\bmell,c$ belonging to a finite subset of $(\CC^r\moins\{0\})\times\CC$, defining $V:=\CC^r\moins\nobreak\bigcup_{\bmell,c}L_{\bmell,c}^{-1}(\ZZ)$
such that, for any $a\in V$, the natural morphism~\eqref{eq:dag*a} is an isomorphism.
\end{theoreme}

The proof will show, more precisely, that the coefficients $\ell_k$ are some combinations with integer coefficients of the residues of $\eta_k$ along $D_i$ (\cf Equation~\eqref{eq:integralcoef}).

\begin{remarque}[Analytic version of Theorem \ref{th:vanishingtwist}]\label{rem:analytic}
Theorem \ref{th:vanishingtwist} is stated in an algebraic setting in order to be compatible with the setting in \cite{Ribeiro24}. However, as~in Part~\ref{part:Euler}, the proof is mainly done in an analytic setting, where $X$ and $Y$ are compact complex manifolds, both are compactifications of $U$, $Z:=X\moins U$ is a closed analytic subset of $X$, and $D:=Y\moins U$ is a divisor with simple normal crossings. We~start from any object $\cM$ in $\catD^\rb_\hol(\cD_X)$, and the morphism \eqref{eq:dag*a} is written as
\[
\cM(!Z)\to\cM(*Z),
\]
where $\cM(*Z)$ is classically denoted by $R\Gamma_{[U]}\cM$, and $\cM(!Z)$ is defined by duality: $\cM(!Z):=\bD\bigl(R\Gamma_{[U]}(\bD\cM)\bigr)$.

The family $\etag$ is a family of closed logarithmic one-forms on $Y$ with poles along~$D$ satisfying the assumption \eqref{eq:vanishingtwist*}. Then the conclusion of Theorem~\ref{th:vanishingtwist} holds (with the same proof) in this analytic setting.
\end{remarque}

\begin{remarque}[Extension to mixed twistor $\cD$-modules]
We keep the analytic setting as in Remark \ref{rem:analytic}. The same theorem holds for mixed twistor $\cD$-modules in the following way (we refer to \cite{Mochizuki11} for this theory and the corresponding notation). Let us assume, for the sake of simplicity, that $Z$ is a hypersurface $H$ in $X$. Let $\cT$ be a mixed twistor $\cD$-module on $X$ that satisfies $\cT=\cT[*H]$. On the other hand, there exists an admissible mixed twistor structure $\cT_{\eta,a}$ on~$(X,H)$ whose associated $\cD_X$\nobreakdash-module is $(\cO_X(*H),\rd+\sum a_k\eta_k)$. Then there exists a uniquely defined mixed twistor $\cD$-module $(\cT\otimes\cT_{\eta,a})[*H]$ (\cf\cite[Prop.\,11.3.3]{Mochizuki11}). The morphism \eqref{eq:dag*a} reads, with the notation of \loccit,
\begin{equation}\label{eq:dag*tw}
(\cT\otimes\cT_{\eta,a})[!H]\to(\cT\otimes\cT_{\eta,a})[*H].
\end{equation}
Since the functor that associates to each twistor $\cD$-module on $X$ its underlying holonomic \DXmodule is conservative, the conclusion of the theorem also applies to~\eqref{eq:dag*tw}.
\end{remarque}

\subsubsection{Applications}\label{subsubsec:appl}
Let $S$ be a smooth complex variety and let $U=\Gm^n$ be the torus of dimension~$n$. We~fix coordinates $t_1,\dots,t_n$ on $U$. For $\lambda\!\in\!(\CC^*)^n$, set $\cN_\lambda\!=\!\bigl(\cO_U,\rd\!+\!(\sum_ka_k\rd t_k/t_k)\bigr)$, with $\exp(-2\pii a_k)=\lambda_k$.

Let $j:U\times S\hto\PP^n\times S$ denote the open inclusion. We can write $J=\bigcup_{k=1}^n\{(k,0),(k,\infty)\}$ corresponding to the divisors $\{t_k=0\}$ and $\{t_k=\infty\}$. We~consider the closed one-forms $\eta_k=\rd t_k/t_k$. They satisfy the corresponding conditions in Theorem \ref{th:vanishingtwist}: for example, $\rd t_k/t_k$ has a nonzero residue only along $\{t_k=0\}$ and $\{t_k=\infty\}$, which do not intersect, so~that $\#I(\eta_k)=1$ if $I$ contains $(k,0)$ or $(k,\infty)$ and is zero otherwise.

 Furthermore, the coefficients $\ell_k$, which are linear combinations with integral coefficients of the residues of $\eta_k$ along the components of $D=\PP^n\moins U$, are integers. It~follows that $V^*:=\exp(-2\pi\sfi V)$ is a Zariski-dense open subset of $(\CC^*)^n$ (in~fact, the complement of a finite union of translated codimension-one subtori of~$U$).

\begin{corollaire}[Local vanishing]
Let $\cM_U$ be a holonomic $\cD_{U\times S}$-module. There exists a Zariski-dense open set $V^*$ of $(\CC^*)^n$ which is the complement of a finite union of translated codimension-one subtori of~$U$, such that for any $\lambda\in V^*$, the natural morphism (\cf \cite[Th.3.2.16]{H-T-T08})
\[
\Dj_!(\cM_U\otimes\cN_\lambda)\to \Dj_*(\cM_U\otimes\cN_\lambda),
\]
is an isomorphism.
\end{corollaire}

\begin{proof}
By taking $X=\PP^n\times S$ and $\eta_k=\rd t_k/t_k$, Theorem \ref{th:vanishingtwist} yields the statement of the corollary.
\end{proof}

\begin{remarque}[Other approaches]
The case when $U=\Gm$ (\ie $n=1$) is proved in \cite{B-G12} by using the theory of $b$-functions. The case where $U=\Gm^n$ ($n\geq1$) should also be obtained from the results in \cite{Maisonobe16} and \cite{Wu21}.
\end{remarque}

\begin{corollaire}[Relative vanishing]
Let $p$ denote the projection $U\times S\!\to\! S$. For~$\cM$ as in the theorem, there exists a Zariski-dense open subset~$V^*$ of~$(\CC^*)^n$, which is the complement of a finite union of translated codimension-one subtori of~$U$, such that for any $\lambda\in V^*$, the natural morphism
\[
\D p_!(\cM_U\otimes\cN_\lambda)\to \D p_*(\cM_U\otimes\cN_\lambda)
\]
is an isomorphism.
\end{corollaire}

\begin{proof}
This is immediate by using the projection $P:\PP^n\times S\to S$ being proper, so that $\D P_!\simeq \D P_*$, and the isomorphisms $\D p_!=\D P_!\circ \Dj_!$, $\D p_*=\D P_*\circ \Dj_*$.
\end{proof}

\section{A criterion for local vanishing}\label{sec:criterion}
We use the setting of Section~\ref{subsec:nota} and the terminology of Section~\ref{subsubsec:goodformalstr}. This section is devoted to the proof of the following criterion.

Let $\cM=\cM(*D_1)$ be a meromorphic flat vector bundle on $(Y,D_1)$ that has a good formal structure along the simple normal crossing divisor $D_1$. Let $d$ be the maximal ramification order of $\cM$ along the irreducible components of $D_1$. Let $D$ be a subdivisor of $D_1$. In particular, we have $\cM=\cM(*D)$.

\begin{proposition}\label{prop:criterion}
Assume that the transverse monodromy of the formal regular part of~$\cM$ along the smooth part of~$D$ has no eigenvalue that is a root of unity of order $\leq d$ (\eg if this formal regular part is zero). Then the natural morphism $\cM(!D)\to\cM=\cM(*D)$ is an isomorphism.
\end{proposition}

For each irreducible component $D_i$ ($i\in J_1$) of $D_1$, ($J_1$ parametrizes the components of $D_1$), there exists a (non-ramified) decomposition $\cM|_{\wh D_i^\circ}\simeq \cM|_{\wh D_i^\circ}^\irr\oplus\cM|_{\wh D_i^\circ}^\reg$, where the second term has regular singularities along~$D_i^\circ$: it is the \emph{formal regular part of $\cM$ along $D_i^\circ$}. Note that, by flatness, the characteristic polynomial of the formal monodromy of $\cM|_{\wh D_i^\circ}^\reg$ around $D_i^\circ$ is constant.

The assumption of the proposition is that, for any $i\in I:=I_1\cap J$ (where $J\subset J_1$ parametrizes the components of $D$), the ramification order $d_i$ needed to obtain a formal decomposition as in Section~\ref{subsubsec:goodformalstr} is bounded by $d$.

\begin{lemme}\label{lem:52}
We fix the local setting as in Section~\ref{subsubsec:goodformalstr}. If the transverse formal monodromy of the regular part of $\cM$ along the smooth part of~$D$ has no eigenvalue that is a root of unity of order $\leq d$, then for each $x_o\in D$ and each $\varphi=x^{-\bme}u(x)\in\Phi_{x_o}$, the monodromy of the germ $\cR_{\varphi,x'_o}$ of the regular meromorphic flat bundle around any component $D'_i$ of $D$ such that $e_i=0$ has no eigenvalue that is equal to one.
\end{lemme}

\begin{proof}
Let us fix $I_1\subset J_1$, $x_o\in D_{I_1}^\circ$, and $x'_o\in D'_{I_1}$. By~the openness property, for any $i\in I_1$ and any $x'\in D^{\prime\circ}_i\cap\nb(x'_o)$, the regular part of $(\D\rho_{\bmd_{I_1}}^*\cM)_{\wh x'}$ is $\bigoplus_\varphi(\cR_\varphi)_{\wh x'}$, where $\varphi=x^{-\bme}u(x)\in\Phi_{x_o}$ is such that $e_i=0$. The assumption implies that, if~$i\in I_1\cap J$, the monodromy of each such $\cR_\varphi$ around $D'_i$ has no eigenvalue that is equal to one (since the ramified monodromy is the $d_i$-th power of the original one).
\end{proof}

\subsubsection{Reduction of Proposition \ref{prop:criterion} to the elementary case}\label{subsubsec:reduc}
We keep the setting of Lemma \ref{lem:52}. In order to check that $\cM(!D)\to \cM\,(=\cM(*D))$ is an isomorphism, it~is enough to do it locally analytically. Assume that we have proved the assertion for the pullback $\cM'$ of $\cM$ by the multi-cyclic ramification $\rho_{\bmd_{I_1}}:(Y',D'_1)\to(Y,D_1)$, so that we still have $\cM'=\cM'(*D'_1)$. As $\cM$ is a direct summand of the pushforward $\D\rho_{\bmd_{I_1}*}\cM'$ (\cf\eg Example 8.7.31 in the \href{https://perso.pages.math.cnrs.fr/users/claude.sabbah/MHMProject/mhm.html}{\textit{MHM project}}), the natural morphism $\cM(!D)\to\cM$ is a direct summand of an isomorphism, hence is also an isomorphism. We~can thus assume that $\cM$ has a good formal \emph{decomposition} along~$D_1$.

As~duality commutes with tensoring with $\cO_{\wh x_o}$, it~is enough to check that, for any $x_o\in D$, the morphism $\cM(!D)_{\wh x_o}\to\cM_{\wh x_o}$ is so, and as co-localization $(!D)$ commutes with tensoring with~$\cO_{\wh x_o}$, we~can start from $\cM_{\wh x_o}$, which is thus decomposed as the direct sum of elementary good formal flat bundles. In~such a case, for each exponential factor $\varphi\in\Phi_{x_o}$, $\cR_\varphi$ is an extension of meromorphic flat bundles of \emph{rank one} with a connection having regular singularities along the germ~$D_{1,x_o}$, so~that it is an extension of rank-one such objects, whose monodromy eigenvalue around~$D_i$ is an eigenvalue of that of~$\cR_\varphi$ around $D_i$. We~are thus left with considering the case of rank-one elementary meromorphic flat bundles as in Section~\ref{subsubsec:rankone}, with $\varphi\in\Phi_{x_o}$; and the assumption of the proposition, together with the lemma above, implies that, for $i\in J$, if~$e_i=0$, the monodromy around~$D_i$ is not equal to one. Note that, at this step, we~can come back to the convergent setting, since the elementary formal objects are convergent.

\subsubsection{Proof in the elementary case}\label{subsubsec:elem}
We work in the local analytic setting where $Y=\Delta^{\ell_1}\times\Delta^{n-\ell_1}$, where $\Delta$ is a small disc in $\CC^n$, and $D_1=\prod_{i=1}^{\ell_1} x_i=0$ (where $x_1,\dots,x_{\ell_1}$ are the coordinates on $\Delta^{\ell_1}$). With the notation of Section~\ref{subsec:nota}, we have set $I_1=\{1,\dots,\ell_1\}$. We~also set $I=\{1,\dots,\ell\}$ with $\ell\leq\ell_1$, corresponding to $D$. We~consider the rank-one meromorphic connection
\[\textstyle
\cN=\bigl(\cO_Y(*D_1),\rd+\rd\varphi+\sum_{i\in I_1}\alpha_i\rd x_i/x_i\bigr),
\]
with
\begin{enumeratea}
\item
$\varphi=u/x^{\bme}$, where $u$ is a holomorphic function on $Y$ that does not vanish anywhere on $Y$ and $x^{\bme}=x_1^{e_1}\cdots x_{\ell_1}^{e_{\ell_1}}$, with $e_i\in\NN$,
\item
$\alpha_i\in\CC$ is such that, for $i\in I$, $(e_i=0\implies\alpha_i\notin\ZZ)$.
\end{enumeratea}
Recall that $D_i^\circ=D_i\moins\bigcup_{i'\neq i}D_{i'}$. If $e_i=0$, $\cN$ has a regular singularity along $D_i^\circ$ and the assumption on $\alpha_i$ means that its monodromy around $D_i^\circ$ is not equal to the identity if $i\in I$. On the other hand, if $e_i\neq0$, then the regular part of $\cN$ along $D_i^\circ$ is zero. We~will prove that $\cN(!D)_{x_o}\to\cN_{x_o}$ is an isomorphism. We~note that $\cN(!D)$ can be obtained as a successive application of the functors~$(!D_i)$ for $i=1,\dots,\ell$. It~is then enough to prove that for each such~$i$, the natural morphism $\cN(!D_i)\to\cN$ is an isomorphism: indeed, let us check by induction on~$k$ that $\cN(!(D_{\leq k}))\to \cN$ is then an isomorphism; we write this morphism as the composition
\[
\cN(!(D_{\leq k-1}))(!D_k)\to \cN(!D_k)\to\cN,
\]
where the first morphism is  obtained by applying the functor $(!D_k)$ to the morphism $\cN(!(D_{\leq k-1}))\to \cN$. We then see by induction that both morphisms are isomorphisms.

It is then enough to check that the jumping indices of the $V$-filtration along each component $D_i$ with $i\in I$ are not integers (\cf\eg\cite[(4.6.6)]{M-S86b}). Set
\[
I_{1,\bme}=\{i\in\{1,\dots,\ell_1\}\mid e_i\neq0\}\quand\nu=\prod_{i\in I_1}x_i^{-m_i}\text{ with $m_i>\alpha_i$}.
\]
Then one checks that $\nu$ is a $\cD_Y$-generator of $\cN$ and, on the other hand, that its Bernstein polynomial along $D_i$ is
\[
\begin{cases}
s-(\alpha_i-m_i)&\text{if }i\notin I_{1,\bme},\\
1&\text{otherwise}.
\end{cases}
\]
The $V$-filtration of $\cN$ along $D_i$ is then that generated by $\nu$, it is constant if $i\in I_{1,\bme}$ and has jumping indices in $\alpha_i+\ZZ$, which is $\neq\ZZ$ if $i\in I\moins I_{1,\bme}$.\qed

\section{Proof of Theorem \ref{th:vanishingtwist}}

\subsubsection*{First part}
We assume that $(Y,D_1)$ is a good compactification of~$U$, that $\cM$ is a \emph{good meromorphic flat bundle} on $(Y,D_1)$, and that $\eta^o_1,\dots,\eta^o_r$ are one-forms on~$U$ whose rational extensions $\eta_1,\dots,\eta_r$ to $Y$ are logarithmic with poles contained in a subdivisor $D=\bigcup_{i\in J}D_i\subset D_1$ and satisfy a \emph{generic} non-vanishing assumption:
\begin{equation}\tag{$\ref{eq:vanishingtwist*}*$}\label{eq:vanishingtwist**}
\begin{cases}
\text{for each $i\in J$, there exists $k\in \{1,\dots,r\}$ such that}\\
\text{the residue of $\eta_k$ along~$D_i^\circ$ is nonzero.}
\end{cases}
\end{equation}%
We will prove the conclusion of Theorem \ref{th:vanishingtwist} under this assumption. We claim that there exists a finite set of affine forms $L_{\bmell,c}(a)=\sum_k\ell_k a_k+c$, with $\bmell,c$ belonging to a finite subset of $(\CC^r\moins\{0\})\times\CC$) such that, if $L_{\bmell,c}(a)\notin\ZZ$ for any $(\bmell,c)$, the assumption of Proposition~\ref{prop:criterion} is fulfilled by $\cM_{\etag,a}$. We define the  dense open set $V\subset\CC^r$ as the complement of the affine hyperplanes $L_{\bmell,c}(\cbbullet)=k$, with $k\in \ZZ$.

Indeed, for each $i\in J$, let $A_i\subset\CC$ be a finite set such that the eigenvalues of the formal monodromy of $\cM$ along $D_i^\circ$ belong to $\exp(-2\pi\sfi A_i)$. Then those of $\cM_{\etag,a}$ along~$D_i^\circ$ take the form \hbox{$\exp\bigl(-2\pi\sfi(\alpha+\sum_{k=1}^ra_k\res_{k,i})\bigr)$} for $\alpha\in A_i$. The condition on~$a$, so that the assumption of Proposition \ref{prop:criterion} is fulfilled for $\cM_{\etag,a}$, reads
\begin{equation}\label{eq:integralcoef}
\alpha+\sum_{k=1}^ra_k\res_{k,i}\notin\tfrac1{d!}\ZZ,\qquad\forall i\in J\text{ and }\forall\alpha\in A_i.
\end{equation}
This defines the desired family of affine forms $L_{\bmell,c}(a)$ on $\CC^r$ provided that, for any $i\in J$, the linear form $L_i:a\mto\sum_{k=1}^ra_k\res_{k,i}$ is not identically zero. This property follows from the non-vanishing property of some $\res_{k,i}$ ($k\in\{1,\dots,r\}$), which is exactly Assumption \eqref{eq:vanishingtwist**}.\qed

\subsubsection*{Second part}
We reduce to the situation considered in the first part. Let $S$ be a closed analytic subset of $Y$ (this will later be the support of $\cM$) and let $S^o$ be a Zariski-dense smooth open subset of $S$. Let $\pi:S'\to Y$ be a proper morphism such that~$S'$ is smooth and
\begin{enumeratea}
\item\label{enum:a}
the image of $\pi$ is $S$ and $\pi:S'\to S$ induces an isomorphism above $S^o$,
\item\label{enum:b}
the pullback $D'_1:=\pi^{-1}(D\cup(S\moins S^o))$ is a simple normal crossing divisor in~$S'$.
\end{enumeratea}

\begin{lemme}
Assume that $\etag^o$ satisfies the assumption \eqref{eq:vanishingtwist*}. Then $\pi^*\etag^o$ satisfies~\eqref{eq:vanishingtwist**}.
\end{lemme}

\begin{proof}
Let $x'_o$ be a smooth point of $D'$. There exist local coordinates in the neighborhood of $x'_o$ and $x_o=\pi(x'_o)$ such that $D'$ is locally defined by $x'_1=0$, $D$ by $x_1\cdots x_\ell=0$, so that we can write $x_i\circ\pi=x_1^{\prime \nu_i}u_i(x'_1,z')$ ($i=1,\dots,\ell$), with $\nu_i\geq1$ and $u_i$ is a local unit. For $k=1,\dots,r$, we have by definition $\eta_k\equiv\sum_{i=1}^\ell \res_{k,i}\rd x_i/x_i\bmod\text{hol.\,forms}$. The pullback $\pi^*\eta_k$ near $x'_o$ reads $\eta'_k=\bigl(\sum_{i=1}^\ell \nu_i\res_{k,i}\bigr)\rd x'_1/x'_1\bmod\text{hol.\,forms}$. 

By assumption \eqref{eq:vanishingtwist*}, there exists $k$ such that, if $\sum_{i=1}^\ell \mu_i\res_{k,i}=0$ with $\mu_i\in\NN$, then all the $\mu_i$ vanish. Since all the $\nu_i$ are $\geq1$, for such a $k$ the property \eqref{eq:vanishingtwist**} holds in a neighborhood of~$x'_o$.
\end{proof}

\subsubsection*{Third part: end of the proof}
We first note that it is enough to prove the assertion of the theorem for holonomic $\cD_U$-modules $\cM_U$: by induction on the amplitude of the complex, it~also holds for any object of $\catD^\rb_\hol(\cD_U)$. We~thus assume that $\cM_U$ is a holonomic $\cD_U$\nobreakdash-mod\-ule. We~consider the extension $\etag$ to $Y$ of the family $\etag^o$.

Let $S$ denote the closure in $Y$ of the support of $\cM_U$, which is a subvariety of~$Y$, let $\iota:S^o\hto U$ be a Zariski-dense open subset of $S$ on which the restriction $\D\iota^*\cM_U$ of~$\cM_U$ is smooth, and let $\pi:S'\to Y$ be a resolution of singularities of~$S$ satisfying~\eqref{enum:a} and~\eqref{enum:b} above. We~can also regard $S^o$ as a Zariski-dense open subset of~$S'$. Since $\D\iota^*\cM_U$ is smooth, we can, furthermore, assume, by the Kedlaya-Mochizuki theorem (\cite[Th.\,8.1.3]{Kedlaya10}, \cite{Mochizuki08}),  that $\D\iota^*\cM_U$ has a good formal structure along $D'=S\moins S^o$.

Set $U'=\pi^{-1}(U)$ and let $\iota':S^o\hto U'$ denote the inclusion. We argue by induction on the dimension of the support of $\cM_U$, so that we can assume that $\cM_U=\D\pi_*\cM_{U'}$ with $\cM_{U'}$ satisfying $\cM_{U'}=\D\iota'_*\D\iota^{\prime*}\cM_{U'}$. In other words, $\cM_{U'}$ is a meromorphic flat bundle on $U'$ with poles on $D'_1\cap U'$, which has a good formal structure along $D'_1$ when extended as a meromorphic flat bundle on $S'$ with poles on $D'_1$.

By the previous lemma, the family $\etag^{\prime o}:=\pi^*\etag^o$ of one-forms on $S'\moins D'$ satisfies the property \eqref{eq:vanishingtwist**} relative to $D'$. Therefore, by the first part of the proof, denoting by $j':U':=\pi^{-1}(U)\hto S'$ the inclusion, we know that
\begin{equation}\label{eq:jprimeMUprime}
\Dj'_!(\cM_{U'})_{\etag^{\prime o},a}\to \Dj'_*(\cM_{U'})_{\etag^{\prime o},a}
\end{equation}
is an isomorphism for $a\in V$, where $V$ is defined as in the first part, relatively to~$\cM_{U'}$ and the family $\etag^{\prime o}$. Applying $\D\pi_*$, we obtain an isomorphism (for $a\in V$):
\[
\D\pi_*\Dj'_!(\cM_{U'})_{\etag^{\prime o},a}\to \D\pi_*\Dj'_*(\cM_{U'})_{\etag^{\prime o},a}.
\]
As $\pi$ is proper, we have functorial isomorphisms $\D\pi_*\Dj'_?\simeq\Dj_?\D\pi_*$ which transform the morphism \eqref{eq:jprimeMUprime} into the natural morphism
\[
\Dj_!\D\pi_*((\cM_{U'})_{\etag^{\prime o},a})\to \Dj_*\D\pi_*((\cM_{U'})_{\etag^{\prime o},a}),
\]
which is thus also an isomorphism. Furthermore, we also have an isomorphism
\[
\D\pi_*((\cM_{U'})_{\etag^{\prime o},a})\simeq(\D\pi_*\cM_{U'})_{\etag^o,a}=(\cM_U)_{\etag^o,a}.
\]
We conclude that, for $a\in V$,
\[
\Dj_!(\cM_U)_{\etag^o,a}\to \Dj_*(\cM_U)_{\etag^o,a}
\]
is an isomorphism, as desired.\qed

\part{Laplace transform of a Stokes-filtered constructible sheaf of exponential type}\label{part:Laplace}

\section{Introduction to Part \ref{part:Laplace}}

\subsubsection{Statement of the theorem}
The Laplace transformation $F_\pm$ with kernel $E^{\pm t\tau}:=(\cO_{\Afu_t\times\Afu_\tau},\rd\pm\rd(\tau t))$ induces an equivalence between the category of algebraic holonomic $\cD_{\Afu_t}$-modules $N$ on the affine line $\Afu_t$ with coordinate $t$, and that of algebraic holonomic $\cD_{\Afu_\tau}$-modules $M$ on the dual affine line $\Afu_\tau$, with quasi-inverse~$F_\mp$. Let us consider the projections
\[
\xymatrix@R=.5cm{
&\Afu_\tau\times\Afu_t\ar[dl]_{q}\ar[dr]^{p}&\\
\Afu_\tau&&\Afu_t
}
\]
so that, for $(N,\nabla)\in\Mod(\cD_{\Afu_t})$ and $(M,\nabla)\in\Mod(\cD_{\Afu_\tau})$, we~have
\begin{align*}
F_\pm(N)&=\D q_*(\D p^*N,\D p^*\nabla\pm\rd(\tau t)),\\
F_\mp(M)&=\D p_*(\D q^*M,\D q^*\nabla\mp\rd(\tau t)).
\end{align*}
It is known (\cf\eg\cite{Malgrange91,K-K-P08}) that this equivalence induces an equivalence between the full subcategories
\begin{itemize}
\item
$\Mod_{\holreg}(\cD_{\Afu_t})$ consisting of regular holonomic $\cD_{\Afu_t}$\nobreakdash-modules $N$, and
\item
$\Mod_{\holexp}(\cD_{\Afu_\tau})$ consisting of algebraic holonomic $\cD_{\Afu_\tau}$\nobreakdash-modules $M$ which have a regular singularity at $\tau=0$, an irregular one of exponential type at $\tau=\infty$, and no other singularity at finite distance.
\end{itemize}
Furthermore, it specializes as an equivalence between the respective sub-categories
\begin{itemize}
\item
$\Mod_{\holreg}^0(\cD_{\Afu_t})$, whose objects $N$ have vanishing global de~Rham hypercohomology $\bR\Gamma(\Afu_t,\pDR(N))$, and
\item
$\Mod_{\holexp}(\cD_{\Afu_\tau},*0)$ whose objects $M$ satisfy $M\simeq M(*0)$. 
\end{itemize}
As in \cite{K-K-P08,Y-Z24}, we~will focus on the latter equivalence. The analytic restriction functor
\[
\Mod_{\holexp}(\cD_{\Afu_\tau},*0)\xrightarrow[\ts\sim]{~\ts i_\infty^{-1}~}\mathsf{Mero}(\infty,\exp)
\]
to the category $\mathsf{Mero}(\infty,\exp)$ of analytic germs at $\infty\in\PP^1_\tau$ of meromorphic connections with an irregular singularity of exponential type, is an equivalence, with a quasi-inverse functor being Deligne's regular meromorphic extension at $\tau=0$. Composing~$F_\pm$ with this restriction functor yields a pair $(F^\infty_\pm,F^\infty_\mp)$ of quasi-inverse functors (\cf Diagram \eqref{eq:main**} below).

For an object $N$ of $\Mod_\holreg(\cD_{\Afu_t})$, the analytic de~Rham complex $\cF=\pDR^\an(N)$ is known to be a $\CC$-perverse complex on~$\CC_t:=\Afuan_t$. From the standard algebraic/analytic comparison theorem $\bR\Gamma(\Afu_t,\pDR(N))\simeq\bR\Gamma(\CC_t,\pDR^\an(N))$, we~deduce that $\pDR(N)$ has zero global hypercohomology if and only if $\pDR^\an(N)$ has zero global hypercohomology, and this is equivalent to the property that $\pDR^\an(N)$ is a constructible sheaf~$\cF$ placed in degree~$-1$, having singularities on a finite subset $C\subset\CC_t$, with vanishing global hypercohomology (\cf \cite{K-K-P08}). We~denote this category as $\Perv^0(\CC_t,C)$, which can be defined with any field of coefficients $\bmk$, and then denoted by $\Perv^0(\CC_t,\bmk,C)$.

On the other hand, the perverse complex $\pDR^\an(M)$ on $\CC_\tau$ can be equipped with a Stokes structure of exponential type at infinity, by means of the functor $(\pDR_{\leq\cbbullet},\pDR_{<\cbbullet})$ (the moderate and rapid decay de~Rham complexes twisted by an exponential factor, \cf \cite[Cor.\,5.4]{Bibi10}), with exponential factors of the form~$c\tau$ for~$c$ belonging to a finite subset $C\subset\CC$. We~denote by $(\cG,\cG_\bbullet)$ this Stokes-$\CC$-perverse sheaf (we~refer to \cite{Bibi10} for this notion), which is an object living on the real blow-up space~$\wt\PP_\tau^1$ of the projective line $\PP^1_\tau$ at infinity (topologically, this is a closed disc), and by $\PervSt(\wt\PP^1_\tau,C)$ the corresponding category. Our assumption that $M$ is equal to $M(*0)$ is equivalent to $\pDR^\an(M)\simeq\bR j_{0*}j_0^{-1}\,\pDR^\an(M)$, where $j_0:\CC^*_\tau\hto\CC_\tau$ denotes the open inclusion. We~denote by $\PervSt(\wt\PP^1_\tau,C,*0)$ the full subcategory of $\PervSt(\wt\PP^1_\tau,C)$ consisting of objects $(\cG,\cG_\bbullet)$ satisfying $\cG\simeq\bR j_{0*}j_0^{-1}\cG$. More precisely, let $\varpi:\wt\PP^1_\tau\to\PP^1_\tau$ denote the real blowing up of the projective line~$\PP^1_\tau$ at $\tau=\infty$ and let us consider the complementary inclusions
\[
\Afu_\tau\Hto{\wtj_\infty}\wt\PP^1_\tau\Hfrom{\wti_\infty}S^1_\tau=\varpi^{-1}(\infty)=\wt\PP^1_\tau\moins\Afu_\tau.
\]
Then $\cH^{-1}\cG|_{\CC^*_\tau}$ is a local system as well as its pushforward by $\wtj_\infty$, so that the restriction of the latter by $\wti_\infty$, which we denote by $\cL$, is a $\CC$-local system on $S^1_\tau$. The latter is equipped with a Stokes filtration of exponential type, thereby defining a Stokes-filtered $\CC$-local system $(\cL,\cL_\bbullet)$ on $S^1_\tau$ indexed by a finite set $C\subset\CC$. We~denote by $\Sto(S^1_\tau,C)$ the corresponding category. For any field of coefficients~$\bmk$, the categories $\PervSt(\wt\PP^1_\tau,\bmk,C)$ and $\Sto(S^1_\tau,\bmk,C)$ are also defined. One easily obtains the equivalence (\cf Section \ref{subsec:equivPervSt})
\begin{equation}\label{eq:equivPervSt}
\PervSt(\wt\PP^1_\tau,*0)\xrightarrow[\ts\sim]{~\ts\wti_\infty^{-1}[-1]~}\Sto(S^1_\tau),
\end{equation}
which can be decorated with the field of coefficients $\bmk$ and the set $C$ corresponding to exponential factors.

The Riemann-Hilbert correspondence and the RH-Birkhoff-Deligne-Malgrange correspondence
\begin{equation}\label{eq:equivcat}
\begin{aligned}
\Mod_{\holreg}^0(\cD_{\Afu_t})&\xrightarrow[\ts\sim]{\ts\pDR^\an}\Perv^0(\CC_t)\\
\Mod_{\holexp}(\cD_{\Afu_\tau},*0)&\xrightarrow[\ts\sim]{\ts(\pDR_{\leq\cbbullet},\pDR_{<\cbbullet})}\PervSt(\wt\PP^1_\tau,*0)
\end{aligned}
\end{equation}
are well-known (\cf\eg\cite[Chap.\,5]{Bibi10} for the second line). We~can also restrict the above categories by specifying the singularity set $C$ (first line) or the index set~$C$ of exponential factors (second line).

\begin{theoreme}\label{th:main}
For any field $\bmk$, there exist pairs of quasi-inverse functors $(F^\top_{\pm},F^\top_{\mp})$ and $(F^{\infty,\top}_{\pm},F^{\infty,\top}_{\mp})$ and a commutative diagram:
\begin{starequation}\label{eq:main*}
\begin{array}{c}
\xymatrix@C=1.5cm{
\Perv^0(\CC_t,\bmk,C)\ar@<.5ex>[r]^-{F^\top_{\pm}}\ar@/_1pc/@<.5ex>[dr]^(.65){F^{\infty,\top}_{\pm}}&\ar@<.5ex>[l]^-{F^\top_{\mp}}\PervSt(\wt\PP^1_\tau,\bmk,C,*0)\ar[d]^{\wti_\infty^{-1}[-1]}_\wr\\
&\Sto(S^1_\tau,\bmk,C)\ar@/^1pc/@<.5ex>[ul]^-{F^{\infty,\top}_{\mp}}
}
\end{array}
\end{starequation}%
which, when $\bmk=\CC$, is compatible, via the Riemann-Hilbert (and RH-Birkhoff-Deligne-Malgrange) correspondence, with the pairs
\begin{starstarequation}\label{eq:main**}
\begin{array}{c}
\xymatrix@C=1.5cm{
\Mod_{\holreg}^0(\cD_{\Afu_t},C)\ar@<.5ex>[r]^-{F_{\pm}}\ar@/_1pc/@<.5ex>[dr]^(.6){F^{\infty}_{\pm}}&\ar@<.5ex>[l]^-{F_{\mp}}\Mod_{\holexp}(\cD_{\Afu_\tau},*0,C)\ar[d]_\wr^{i_\infty^{-1}}\\
&\mathsf{Mero}(\infty,C)\ar@/^1pc/@<.5ex>[ul]^-{F^{\infty}_{\mp}}
}
\end{array}
\end{starstarequation}%
via the equivalences \eqref{eq:equivcat}.
\end{theoreme}

In \cite{Bibi10}, the correspondence $\cF\mto(\cG,\cG_\bbullet)$, called \emph{topological Laplace transformation}, has been made explicit for any perverse complex $\cF$ with coefficients in a field~$\bmk$, and it has been shown that, when $\bmk=\CC$, it is compatible via the Riemann-Hilbert (and RH-Birkhoff-Deligne-Malgrange) correspondence with the Laplace transformation of regular holonomic $\cD_{\Afu_t}$-modules. Making explicit the inverse correspondence and its compatibility with \eqref{eq:equivcat} has been motivated by the work \cite{Y-Z24}. This text can then be regarded as a complement to \cite[Chap.\,7]{Bibi10}. Let us also mention that the construction of the pair $(F^{\infty,\top}_{\pm},F^{\infty,\top}_{\mp})$ is a very special case of the more general construction done by T.\,Mochizuki in \cite{Mochizuki18}, and was initiated by B.\,Malgrange in \cite{Malgrange91}.

\subsubsection{Comparison with the approach of \cite{Y-Z24}}
In \cite{Y-Z24}, Tony Yue Yu and Shaowu Zhang prove a formula for the inverse topological Laplace transformation $(\cG,\cG_\bbullet)\mto\cF$ under the global vanishing assumption that we also adopt here. Let us emphasize that the main nice idea of \cite{Y-Z24} is to use co-Stokes structures, which are easier to handle. We~propose here an approach to their formula that we show to be compatible with the Riemann-Hilbert correspondence (in~\cite{Y-Z24}, this compatibility is obtained in an indirect way). We~make full use of the notion of Stokes structure in dimension $>1$ (namely, in~dimension two).

In \loccit, the authors first replace the Stokes-filtered local system $(\cL,\cL_\bbullet)$ on~$S^1_\tau$ with a co-Stokes-filtered local system $(\cL,\cL_<)$ on~$S^1_\tau$ (\ie by replacing the order~$\leq$ with the strict order $<$)---both notions are known to be equivalent---, and then ``thicken'' this object to produce, for any given $t\in\CC^*_t$, a similar object on $\CC^*_\tau$. This is motivated by the fact that the topological Laplace transform of a perverse complex on $\CC_t$ having no global hypercohomology (equivalently, of a constructible sheaf placed in degree one on $\CC_t$ having no global cohomology, \cf\cite{K-K-P08}) can be defined in order to produce such a ``thickened'' object. The drawback of this approach is that the Stokes-filtered de~Rham complex attached to a free $\CC[\tau,\tau^{-1}]$-module with connection having a regular singularity at $\tau=0$ and an irregular one of exponential type at $\tau=\infty$ produces a ``non-thickened'' Stokes constructible sheaf. In~order to obtain a Riemann-Hilbert correspondence, the authors of \cite{Y-Z24} make use of the Riemann-Hilbert correspondence for the corresponding regular holonomic module obtained by Laplace transform, and the property that Laplace transformation has an inverse transformation both at the topological level (a~property that they prove) and at the algebraic/analytic level (a~property much easier to prove).

In the presentation we give in Section \ref{sec:RH} of this note, the Riemann-Hilbert correspondence is obtained independently of the $\pm$ involutivity of the topological Laplace transformation. For that purpose, we~make full use of the notion of a Stokes-filtered local system along a divisor with normal crossings.

\subsubsection{Basic lemmas}
We will make use of the following basic lemmas.

\begin{lemme}\label{lem:vanishing}
Let $I$ be a non-empty semi-closed interval in $\RR$ and let $\ov\Delta_<$ be a non-empty closed disc with a closed interval deleted in its boundary. Then, for any field~$\bmk$, we~have $H_\rc^*(I,\bmk)=0$ and $H_\rc^*(\ov\Delta_<,\bmk)=0$.\qed
\end{lemme}

\begin{lemme}\label{lem:recol}
Let $j:U\hto Y$, \resp $i:F\to Y$, be complementary open, \resp closed, inclusions of topological spaces. Given a sheaf $\cH_U$ on $U$ and a subsheaf $\cH_F\subset i^{-1}j_*\cH_U$ on~$F$, there exists a unique subsheaf $\cH$ on $Y$ such that $j^{-1}\cH\!=\!\cH_U$ and \hbox{$i^{-1}\cH\!=\!\cH_F$}.\qed
\end{lemme}

\Subsubsection{Dramatis personae}
\begin{itemize}
\item
The holonomic $\Cltau$-module $(M,\nabla)$, its analytic germ $\cM$ at $\tau=\infty$, which is a finite-dimensional $\CC\{\tau'\}[\tau^{\prime-1}]$-vector space with connection (with $\tau':=1/\tau$), and its Stokes-filtered local system $(\cL,\cL_\bbullet)$ on $S^1_\tau$ at $\tau=\infty$ indexed by $C\tau$ with $C\subset\CC$. Here, $S^1_\tau$ is the fiber $\varpi^{-1}(\infty)$ of the real blowing up $\varpi:\wt\PP^1_\tau\to\PP^1_\tau$ of $\PP^1_\tau$ at $\tau=\infty$.
\item
The exponentially twisted $\CC[\tau,t]\langle\partial_\tau,\partial_t\rangle$-module $(\D q^*M,\nabla-\rd(\tau t))$ and its analytic germ~$\scrM$ along $D_\infty=\{\infty\}\times\CC_t$, which is a germ along $D_\infty$ of a meromorphic flat bundle with poles along $D_\infty$.
\item
The Stokes-filtered local system $(\scrL,\scrL_\bbullet)$ on $S^1_\tau\times\CC_t$ indexed by \hbox{$\Phi\!=\!\{\varphi_c\mid c\!\in\! C\}$} with $\varphi_c(\tau,t)=\tau(c-t)$; it is obtained by an exponential twist from $(\cL,\cL_\bbullet)$; the main character is played by the subsheaf $\scrL_{<0}\subset\scrL$.
\item
The complex blow-up $e:X\to\PP^1_\tau\times\CC_t$ at the points $(\infty,c)$, the divisor with normal crossings $D=e^{-1}(D_\infty)$ and the real blow-up $\varpi_X:\wt X\to X$ along the irreducible components of $D$, with boundary $\wt D$; the pullback $\scrM_\sX$ of $\scrM$ by~$e$, which is a germ along $D$ of a meromorphic flat bundle with poles along $D$, and thus a germ of a left $\cD_X$-module.
\item
The Stokes-filtered local system $(\scrL_\sX,\scrL_{\sX,\bbullet})$ on $\wt D$ pulled back from $(\scrL,\scrL_\bbullet)$ by $\wt e:\wt X\to\wt\PP^1_\tau\times\CC_t$ induced by $e$, indexed by $\Phi_X=\{\psi_c:=\varphi_c\circ e\mid c\in C\}$; the main character is $\scrL_{\sX,\leq0}$ and the main relation, which will be proved, is $\scrL_{<0}=\bR\wt e_*\scrL_{\sX,\leq0}$.
\item
On the Fourier side, the regular holonomic $\Clt$-module $N$ and its associated perverse complex $\cF=\pDR^\an(N)$.
\item
If $N$ is the Fourier transform of $M$, we prove the corresponding topological relation $\cF\simeq \bR\wt p_*\scrL_{<0}\simeq\bR\wt\pi_*\scrL_{\sX,\leq0}$ (\cf Diagrams \eqref{eq:diag1} and \eqref{eq:diag2}).
\end{itemize}

\section{The Laplace transform of Stokes-filtered local systems of\nobreakspace exponential\nobreakspace type}\label{sec:FStokes}

Let $\bmk$ be a field. Let us fix the parametrization
\begin{align*}
\RR/2\pi\ZZ&\to S^1\subset\CC\\
\theta&\mto\rme^{\sfi\theta}
\end{align*}
of $S^1$, and we often write $\theta$ instead of $\theta\bmod2\pi$. We~will implicitly identify $\RR/2\pi\ZZ$ with $S^1$ via this exponential map. Recall (\cf\cite{Bibi10}) that a $\bmk$-Stokes-filtered local system $(\cL,\cL_\bbullet)$ of exponential type on $S^1$ indexed by a finite set $C\subset\CC$ is a locally constant sheaf $\cL$ of $\bmk$-vector spaces on $S^1$ equipped with a family (indexed by $C$) of nested subsheaves $\cL_{<c}\subset\cL_{\leq c}$ of $\bmk$-vector spaces such that
\begin{itemize}
\item
for each $c\in C$, $\gr_c\!\cL:=\cL_{\leq c}/\cL_{<c}$ is a $\CC$-local system on $S^1$,
\item
for $c'\neq c$ and $\theta\in \RR/2\pi\ZZ$, we~have $\cL_{\leq c',\theta}\subset\cL_{<c,\theta}$ if $\reel[(c'-c)e^{-\sfi\theta}]<0$ (a~property that we denote by $c'<_\theta c$, and we set $c'\leq_\theta c$ if $c'=c$ or $c'<_\theta c$),
\item
for any $\theta\in \RR/2\pi\ZZ$, we~have $\cL_{c,\theta}\simeq\bigoplus_{c'\leq_\theta c}\gr_{c'}\!\cL_\theta$.
\end{itemize}

The filtration indexed by $C$ can be extended to a filtration indexed by $\CC$ by setting, for each $t\in\CC$ and each $\theta\in\RR/2\pi\ZZ$,
\[
\cL_{\leq t,\theta}=\sum_{\substack{c\in C\\ c\leq_\theta t}}\cL_{\leq c,\theta}\quand \cL_{<t,\theta}=\sum_{\substack{c\in C\\ c<_\theta t}}\cL_{\leq c,\theta},
\]
and one checks that these formulas define nested subsheaves $\cL_{<t}\subset\cL_{\leq t}$ of $\cL$, which are equal if $t\notin C$.

The set of \emph{Stokes directions of a finite set $C\subset\CC$} is the subset of $S^1$ consisting of the points $e^{\sfi(\pm\pi/2-\arg(c_i-c_j))}$ for each pair $c_i\neq c_j$ in $C$. In~the continuation of this section, we~often omit the reference to the field $\bmk$, which is fixed. We~consider the following categories:
\begin{itemize}
\item
the category $\Sto(S^1,\bmk,C)$ of Stokes-filtered $\bmk$-local systems $(\cL,\cL_\bbullet)$ of exponential type on $S^1$ indexed by $C$;
\item
the category $\Perv^0(\CC_t,\bmk,C)[-1]$ of $\CC$-constructible sheaves $\cF$ of $\bmk$-vector spaces on $\CC_t$ with singularities contained in $C$ and satisfying $H^*(\CC_t,\cF)=0$. The notation $\Perv^0(\CC_t,\bmk,C)[-1]$ is justified by the property (\cf\cite{K-K-P08}) that $\cF$ is an object of this category if and only if $\cF[1]$ is a perverse sheaf on $\CC_t$ with vanishing global hypercohomology. 
\end{itemize}

The main result of this section consists of the definition of the functor $F^{\infty,\top}_-$ (Definition \ref{def:Finftop}), in~such a way that, together with the functor $F^{\infty,\top}_+$ constructed in \cite[Chap.\,7]{Bibi10}, we~obtain a diagram
\begin{equation}\label{eq:Ftop}
\xymatrix@C=1.5cm{
\Perv^0(\CC_t,\bmk,C)\ar@<.5ex>[r]^-{F^{\infty,\top}_+}&\ar@<.5ex>[l]^-{F^{\infty,\top}_-}\Sto(S^1,\bmk,C)
}
\end{equation}
which satisfies the next proposition, proved in Section \ref{subsec:proofLaplacelocsyst}:

\begin{proposition}\label{prop:Laplacelocsyst}
The functors $F^{\infty,\top}_+,F^{\infty,\top}_-$ are quasi-inverse equivalences of categories.
\end{proposition}

A similar result holds by exchanging $+$ and $-$ in \eqref{eq:Ftop}.

\Subsubsection{A preliminary computation for Stokes-filtered local systems of exponential type}

\begin{lemme}\label{lem:HkFL}
Let $(\cL,\cL_\bbullet)\in\Sto(S^1,C)$. For each $t\in\CC$, we~have $H^k(S^1,\cL_{<t})=0$ if $k\neq1$ and
\[
\dim H^1(S^1,\cL_{<t})=
\begin{cases}
\rk\cL&\text{if }t\notin C,\\
\rk\cL-\rk\gr_c\!\cL&\text{if }t=c\in C.
\end{cases}
\]
\end{lemme}

\begin{proof}
Given $t\in \CC$, let us decompose $S^1$ as $I_1\cup I_2$, where each $I_i$ is a closed interval of length $\pi+\epsilon$ with $\epsilon>0$ small enough such that each $I_i$ contains exactly one Stokes direction of $C\cup\{t\}$, and its endpoints are not Stokes directions. We~denote the inclusion $I_i\hto S^1$ by $\iota_i$, and $\iota_{12}:I_1\cap I_2\hto S^1$. For any complex of sheaves $\cF$ on~$S^1$ we have the Mayer-Vietoris triangle
\[
\cF\to R\iota_{1*}\,\iota_1^{-1}\cF\oplus R\iota_{2*}\,\iota_2^{-1}\cF\to R\iota_{12*}\,\iota_{12}^{-1}\cF\To{+1}.
\]
By considering the associated long exact sequence attached to $R\Gamma(S^1,\cbbullet)$ of the above triangle with $\cF=\cL_{<t}$, we~are left with showing
\begin{itemize}
\item
$H^k(I_i,\cL_{<t})=0$ for any $k$,
\item
$H^k(I_{12},\cL_{<t})=0$ for $k\geq1$ and $\dim H^0(I_{12},\cL_{<t})=\rk\cL$ if $t\notin C$, or~$\rk\cL-\rk\gr_c\!\cL$ if $t=c\in C$.
\end{itemize}

On each interval $I_i=(\theta_i,\theta'_i)$, $(\cL,\cL_\bbullet)$ is graded (\cf\cite[Lem.\,3.12]{Bibi10}). By~considering each homogeneous component separately, we~can thus assume that $(\cL,\cL_\bbullet)|_{I_i}$ has only one jumping index $c_o\in\CC$. This means that $\cL|_{I_i}\!=\!\gr_{c_o}\!\cL_{I_i}\!\simeq\!\CC_{I_i}$, and $\cL_{<t}=\gr_{c_o}\!\cL$ on the open sub-interval of $I_i$ where $c_o<t$ and~$0$ otherwise.

If $t\neq c_o$, this interval $I_{i,t}$ is non-empty and takes the form $(\theta_i,\theta_t)$ (or $(\theta_t,\theta'_i)$). Denoting by $j:I_{i,t}\hto I_i$ the inclusion, we conclude that $\cL_{<t}\simeq j_!j^{-1}\CC_{I_i}$, and the first point follows from the vanishing $H^*_\rc([0,1),\bmk)=0$ (\cf Lemma \ref{lem:vanishing}). For the second point, this description shows that $\cL_{<t}|_{I_{12}}$ is a constant sheaf and it is zero on one connected component of $I_{12}$ and equal to $\cL$ on the other component.

If $t=c_o$, we~have identically $(\gr_{c_o}\!\cL)_{<t}=0$ and there is nothing to prove.
\end{proof}

\begin{remarque}
The cohomology $H^k(S^1,\cL_{\leq c})$, for $c\in C$, behaves a little differently. Indeed, in the decomposition considered in the last part of the proof of the lemma, the component of $\cL_\bbullet$ corresponding to $c\in C$ is a locally constant sheaf on $S^1$ isomorphic to $\gr_c\!\cL$, and we have a non-canonical decomposition
\[
H^1(S^1,\cL_{\leq c})\simeq H^1(S^1,\cL_{<c})\oplus H^1(S^1,\gr_c\!\cL),\quad H^0(S^1,\cL_{\leq c})\simeq H^0(S^1,\gr_c\!\cL).
\]
\end{remarque}

\subsubsection{Topology of the family of real blow-ups}
For any $a\in\CC^*$, we~define the open intervals $S^1_{0<a},S^1_{0>a}\subset S^1$ by
\begin{align*}
S^1_{0<a}&=\{\theta\mid\reel(a\rme^{-\sfi\theta})>0\}=\{\theta\mid\arg a-\theta\in(-\pi/2,\pi/2)\bmod2\pi\},\\
S^1_{a<0}&=\{\theta\mid\reel(a\rme^{-\sfi\theta})<0\}=\{\theta\mid\arg a-\theta\in(\pi/2,3\pi/2)\bmod2\pi\},
\end{align*}
and the respective complementary closed intervals
\begin{align*}
S^1_{a\geq 0}&=\{\theta\mid\reel(a\rme^{-\sfi\theta})\leq0\}=\{\theta\mid\arg a-\theta\in[\pi/2,3\pi/2]\bmod2\pi\},\\
S^1_{0\leq a}&=\{\theta\mid\reel(a\rme^{-\sfi\theta})\geq0\}=\{\theta\mid\arg a-\theta\in[-\pi/2,\pi/2]\bmod2\pi\}.
\end{align*}
If $a=0$, then by convention $S^1_{0<a}=S^1_{0>a}=\emptyset$ and $S^1_{0\leq a}=S^1_{0\geq a}=S^1$.

In the product $\cS^1:=S^1\!\times\CC_t$, we~define for each $s\in\CC$ the subsets
\begin{equation}\label{eq:ccS^1}
\cS^1_{t<s}=\bigsqcup_{t\in\CC_t}S^1_{0<s-t},\qquad \cS^1_{t>s}=\bigsqcup_{t\in\CC_t}S^1_{s-t<0},
\end{equation}
and $\cS^1_{t\leq s},\cS^1_{t\geq s}$ similarly. The next lemma is immediate.

\begin{lemme}
The subsets $\cS^1_{t<s},\cS^1_{t>s}$ are open in $S^1\!\times\CC_t$.\qed
\end{lemme}

For $s$ fixed, the fiber $\cS^1_{t<s,\theta}$ over $\theta\in S^1$ is the open half-plane $\reel[(t-s)\rme^{-\sfi\theta}]<0$ (bounded by the real line $\reel[(t-s)\rme^{-\sfi\theta}]=0$ going through $s$), while $\cS^1_{t\geq s,\theta}$ is the complementary closed half-plane.

Let $\wt\PP^1_t=\CC_t\sqcup S^1_t$ be the real blow-up of $\PP^1$ at $t=\infty$. Let $t'=1/t$ denote the coordinate at $\infty$ on $\PP^1_t$. The coordinate on $S^1_t$ is $-\arg t'=\arg t$ if we choose the orientation induced by that of $\CC_t$. For any given $s\in\CC$ and $t$ with $|t|\gg0$, we~have $\arg(t-s)=\arg t$ on $S^1_t$ since $1/(t-s)=t'\cdot (1/(1-st'))$. We~consider the open subsets of $S^1\!\times S^1_t$ defined by
\begin{equation}\label{eq:wtccS^1>}
\begin{aligned}
\wt\cS^1_{\infty,<}&=\{(\theta,\arg t)\in S^1\!\times S^1_t\mid\arg t-\theta\in(\pi/2,3\pi/2)\bmod2\pi\},\\
\wt\cS^1_{\infty,>}&=\{(\theta,\arg t)\in S^1\!\times S^1_t\mid\arg t-\theta\in(-\pi/2,\pi/2)\bmod2\pi\},
\end{aligned}
\end{equation}
and we~extend $\cS^1_{t<s},\cS^1_{t>s}$ as open subsets of $S^1\!\times\wt\PP^1_t$ by setting
\begin{equation}\label{eq:wtccS^1}
\wt\cS^1_{t<s}=\cS^1_{t<s}\sqcup\wt\cS^1_{\infty,<},\qquad
\wt\cS^1_{t>s}=\cS^1_{t>s}\sqcup\wt\cS^1_{\infty,>}.
\end{equation}
We~define similarly the respective complementary closed subsets $\wt\cS^1_{t\geq s}$ and $\wt\cS^1_{t\leq s}$ of $S^1\!\times\wt\PP^1_t$, which are equal to $\cS^1_{t\geq s}\sqcup\wt\cS^1_{\infty,\geq}$ and $\cS^1_{t\leq s}\sqcup\wt\cS^1_{\infty,\leq}$ respectively, with an obvious notation. We~have an exact sequence\footnote{For any locally closed subset $Z\subset X$ of a topological space $X$ and any sheaf $\cF$ on $X$, $\cF_Z$ denotes the restriction of $\cF$ to $Z$ extended by zero.}
\begin{equation}\label{eq:exactS}
0\to\bmk_{\wt\cS^1_{t<s}}\to\bmk_{S^1_\tau\times \wt\PP^1_t}\to\bmk_{\wt\cS^1_{t\geq s}}\to0.
\end{equation}
Denote by $i_s:S^1\!\times\{s\}\hto S^1\!\times\wt\PP^1_t$ the inclusion. Then we have the equalities
\begin{equation}\label{eq:is}
i_s^{-1}\bmk_{\wt\cS^1_{t<s}}=0\quand i_s^{-1}\bmk_{\wt\cS^1_{t\geq s}}=\bmk_{S^1\!\times\{s\}}.
\end{equation}
We~also set, for example,
\begin{equation}\label{eq:S1<}
\begin{aligned}
\wt\cS^1_<&:=(S^1\!\times\CC_t)\sqcup\wt\cS^1_{\infty,<}\underset{\text{open}}\subset S^1\!\times\wt\PP^1_t,\\ \wt\cS^1_\geq&:=(S^1\!\times\CC_t)\sqcup\wt\cS^1_{\infty,\geq}\underset{\text{closed}}\subset S^1\!\times\wt\PP^1_t.
\end{aligned}
\end{equation}
Therefore, $\wt\cS^1_{t<s}\subset\wt\cS^1_<$ for each $s\in \CC$.

If we identify $\CC_t$ with an open disc centered at $s$, the fiber $\cS^1_{t<s,\theta}$, which is an open half-plane, is identified with an open half-disc. Then $\wt\cS^1_{t<s,\theta}$ is obtained by adding to it the open interval of $S^1_t$ consisting of points $\arg t$ satisfying $\cos(\arg t-\theta)<0$. We~have thus obtained the following topological description.

\begin{lemme}\label{lem:wtccS^1}
Let us fix $s\in\CC$. Then, for any $\theta\in S^1$, the fibers $\wt\cS^1_{t<s,\theta}$ and $\wt\cS^1_{t>s,\theta}$ are homeomorphic to a closed disc with one non-empty closed interval deleted in its boundary. Furthermore, for any $s\neq s'\in\CC$, the fiber $\wt\cS^1_{s'<t<s,\theta}:=\wt\cS^1_{s'<t,\theta}\cap\wt\cS^1_{t<s,\theta}$ is
\begin{itemize}
\item
homeomorphic to a closed disc with two disjoint non-empty closed intervals deleted in its boundary, if $s'<_\theta s$, and
\item
is empty otherwise.\qed
\end{itemize}
\end{lemme}

\subsubsection*{Cohomological computations}
As an immediate consequence, we~can compute the cohomology with compact support of these sets. We~find, according to Lem\-ma~\ref{lem:vanishing} or a variant of it:
\begin{equation}\label{eq:wtccS^1cohom}
\begin{aligned}
\bR\Gamma_\rc(\wt\cS^1_{t<s,\theta},\bmk)&=0,\\
\bR\Gamma_\rc(\wt\cS^1_{s'<t<s,\theta},\bmk)&=
\begin{cases}
\bmk[-1]&\text{if }s'<_\theta s,\\
0&\text{otherwise}.
\end{cases}
\end{aligned}
\end{equation}

Let $\wt p:S^1\!\times\wt\PP^1_t\to\wt\PP^1_t$ denote the projection.

\begin{lemme}\label{lem:wtccS^1cohom}
We have
\[
\wt p_*\bmk_{\wt\cS^1_{t\geq s}}\simeq\bmk_{\wt\PP^1_t},\quad
R^1\wt p_*\bmk_{\wt\cS^1_{t\geq s}}\simeq\bmk_s,\quad
R^k\wt p_*\bmk_{\wt\cS^1_{t\geq s}}=0,\text{ for } k\geq2.
\]
\end{lemme}

\begin{proof}
All fibers of $\wt p:\wt\cS^1_{t\geq s}\to\wt\PP^1_t$ are connected, hence the first point. The fiber of~$p$ at~$t$ is a non-empty closed interval of $S^1$ if $t\neq s$, while the fiber over $t=s$ is equal to~$S^1$, yielding the last two points.
\end{proof}

\subsubsection{The Laplace transform of a Stokes-filtered local system of exponential type}\label{subsec:FStokes}
In order to motivate and justify the construction, we~let $\CC_\tau$ be the complex line with coordinate $\tau$, we let $\PP^1_\tau$ denote its projective completion, and we let $\varpi:\wt\PP^1_\tau\to\PP^1_\tau$ denote the real blowing up of $\PP^1_\tau$ at $\tau=\infty$, with $\varpi^{-1}(\infty):=S^1_\tau\simeq S^1$. When needed, we~use the coordinate $\tau'$ on the chart centered at $\infty$, so that $\tau'=1/\tau$ on $\CC^*_\tau$. We~mainly use the polar coordinates for~$\tau'$ on $S^1\!\times(0,\infty]$, that we denote $(\rho_{\tau'},\theta)$, \ie $\tau'=\rho_{\tau'}e^{\sfi\theta}$.

We~consider the product $\PP^1_\tau\times\CC_t$ and the commutative diagram:
\begin{equation}\label{eq:diag1}
\hspace*{2.5cm}\begin{array}{c}
\xymatrix@R=.5cm{
S^1_\tau\ar@{^{ (}->}[d]\ar@/_2.1pc/[dddd]&\wt D_\infty\ar[l]_-{\wt q}
\ar@{_{ (}->}[d]
\ar@/^2.1pc/[dddd]|(.33)\hole|(.69)\hole&\wt D_\infty:=S^1_\tau\times\CC_t
\\
\wt\PP^1_\tau\ar[dd]_-{\varpi}&\wt\PP^1_\tau\times\CC_t\ar[l]_-{\wt q}\ar[dd]_-{\varpi\times\id}\ar[rd]^-{\wt p}\\
&&\CC_t\\
\PP^1_\tau&\PP^1_\tau\times\CC_t\ar[l]_-{q}\ar[ru]_-{p}\\
\{\infty\}\ar@{_{ (}->}[u]&D_\infty\ar@{^{ (}->}[u]\ar[l]_-{q}&D_\infty:=\{\infty\}\times\CC_t
}
\end{array}
\end{equation}
In this diagram, $q,p$ are the natural projections of $\PP^1_\tau\times\CC_t$ to the first and second factor respectively, and $\wt q,\wt p$ are the corresponding projections from $\wt\PP^1_\tau\times\CC_t$. Since $S^1_\tau=\varpi^{-1}(\infty)\subset\wt\PP^1_\tau$, we~also have $\wt D_\infty=(\varpi\times\id)^{-1}(D_\infty)\subset\wt\PP^1_\tau\times\CC_t$.

We denote by $\Phi$ the set of meromorphic functions $\varphi_c(\tau',t)=\tau(c-t)=(c-t)/\tau'$ with $c\in C$. The common pole divisor of these functions is $D_\infty$. The set $\Phi$ is a \emph{good set of meromorphic functions with poles along $D_\infty$}, since for any pair $c\neq c'\in C$, the difference $\varphi_c-\varphi_{c'}=(c-c')/\tau'$ is purely monomial, in the sense of \hbox{\cite[Def.\,9.8]{Bibi10}}. Note that, however, the  family $\Phi\cup\{0\}$ is \emph{not good} at the points $(\infty,c)$ for $c\in C$ (indeed, for a given $c\in C$, $\varphi_c$~is not purely monomial at $t=c$, in the sense of \hbox{\cite[Def.\,9.8]{Bibi10}}).

We define a Stokes-filtered local system $(\scrL,\scrL_\bbullet)$ indexed by $\Phi$ on $\wt D_\infty$ by setting, for $\varphi_c\in\Phi$:
\[
\scrL=\wt q{}^{-1}(\cL),\quad
\scrL_{\leq\varphi_c}=\wt q{}^{-1}(\cL_{\leq c})\quand \scrL_{<\varphi_c}=\wt q{}^{-1}(\cL_{< c}).
\]
That $(\scrL,\scrL_\bbullet)$ is a Stokes-filtered local system is a consequence of the property that, on the one hand, $\gr_{\varphi_c}\scrL=\wt q^{-1}\gr_c\!\cL$, which is a locally constant sheaf, and on the other hand, for any $(\theta,t)\in S^1_\tau\times\CC_t$, we~have $\varphi_c\leq_{(\theta,t)}\varphi_{c'}$ if and only if $c\leq_{\theta}c'$.

In other words, for a fixed $t\in\CC$, we~have twisted the exponential factors of $(\cL,\cL_\bbullet)$ by adding $-t/\tau'$ to all of them without changing the corresponding sheaves~$\cL_\bbullet$.

\subsubsection*{The sheaf $\scrL_{<0}$}\label{subsec:Lneg}
Although the  family $\Phi\cup\{0\}$ is not good at the points $(\infty,c)$ for $c\in C$, the subset $\{(\theta,t)\mid\varphi_c<_{(\theta,t)}0\}$ is open in $\wt D_\infty$, as~it contains no point of $S^1_\tau\times\{c\}$. The latter property yields that, if we set for each $(\theta,t)\in S^1_\tau\times\CC_t$,
\[
\scrL_{<0,(\theta,t)}:=\sum_{c\mid\varphi_c<_{(\theta,t)}0}\scrL_{<\varphi_c,(\theta,t)}=\bigoplus_{\varphi_c<_{(\theta,t)}0}\gr_c\!\cL_{\theta},
\]
then $\scrL_{<0,(\theta,t)}$ is a germ at $(\theta,t)$ of a well-defined sheaf $\scrL_{<0}$ on $\wt D_\infty$.

\begin{lemme}\label{lem:Lneg}
The sheaf $\scrL_{<0}$ is the unique subsheaf of $\scrL=\wt q^{-1}\cL$ whose restriction to $S^1_\tau\times\{t\}$ is equal to $\cL_{<t}$ for any $t\in\CC_t$.
\end{lemme}

\begin{proof}
For $(\theta,t)$ fixed, we~have $\{c\mid\varphi_c<_{(\theta,t)}0\}\!=\!\{c\mid c<_\theta t\}$, and therefore $\scrL_{<0,(\theta,t)}\!=\!\cL_{<t,\theta}$, as~desired. On the other hand, two subsheaves $\scrL',\scrL''$ of $\scrL=\wt q^{-1}\cL$ coincide if (and only if) their restrictions to $S^1_\tau\times\{t\}$ coincide for each $t$: denoting by $i_t$ the inclusion of the $t$-fiber, and noting that the functor $i_t^{-1}$ is exact, the assumption is that $i_t^{-1}[(\scrL'+\scrL'')/\scrL']=0$ for each $t$, and thus
\[
(\scrL'+\scrL'')/\scrL'=0.\qedhere
\]
\end{proof}

\begin{proposition}[Pushforward of $\scrL_{<0}$]\label{prop:pushLneg}
The complex $\bR\wt p_*\scrL_{<0}$ has nonzero cohomology in degree one at most, and $R^1\wt p_*\scrL_{<0}$ is a constructible sheaf on $\CC_t$ with singularities contained in~$C$. Furthermore, its global cohomology $H^*(\CC_t,R^1\wt p_*\scrL_{<0})$ is zero.
\end{proposition}

As a consequence (\cf \cite{K-K-P08}), $\bR\wt p_*\scrL_{<0}[2]$ is a perverse complex.

\begin{definition}[\cf\cite{Y-Z24}]\label{def:Finftop}
The topological Laplace transform $F^{\infty,\top}_-(\cL,\cL_\bbullet)$ of the Stokes-filtered local system $(\cL,\cL_\bbullet)$ is the object $\bR\wt p_*\scrL_{<0}[2]$ of $\Perv^0(\CC_t,\bmk,C)$.
\end{definition}

\begin{proof}[Proof of the first part of Proposition \ref{prop:pushLneg}]
As $\wt p$ is proper, we~can argue fiber by fiber. It~is easy to check that $R^k\wt p_*\scrL_{<0}$ is a locally constant sheaf on $\CC_t\moins C$ for any~$k$. The first part of the proposition is a then consequence of Lemma \ref{lem:Lneg} together with Lemma \ref{lem:HkFL}.
\end{proof}

\begin{proof}[Proof of the second part of Proposition \ref{prop:pushLneg}]
According to the first part, we~have to prove the vanishing of the hypercohomology of $\bR\wt q_*\bR\wt p_*\scrL_{<0}$ on $S^1_\tau\times\CC_t$ and, by first switching the pushforward functors, it is enough to prove the vanishing of $\bR\wt q_*\scrL_{<0}$. In~order to argue fiberwise, we~need to replace $\wt q:S^1_\tau\times\CC_t\to S^1_\tau$ with a proper morphism. We~thus compactify $\CC_t$ by the real blow-up $\wt\PP^1_t$ of $\PP^1_t$ at $t=\infty$ and we denote by $S^1_t$ the complement of $\CC_t$ in $\wt\PP^1_t$. Let $j_t:S^1_\tau\times\CC_t\hto S^1_\tau\times\wt\PP^1_t$ denote the open inclusion.

Let $t'=1/t$ be the coordinate of $\PP^1_t$ at infinity. Near $(\infty,\infty)$, each $\varphi_c$ reads $c/\tau'-1/t'\tau'$ and $\Phi\cup\{0\}$ is good at this point. We~note that the set
\[
\{(\theta,\theta_t)\in S^1_\tau\times S^1_t\mid\varphi_c<_{(\theta,\theta_t)}0\}=\{(\theta,\theta_t)\mid\theta_t-\theta\in(-\pi/2,\pi/2)\bmod2\pi\}=\wt\cS^1_{\infty,>}
\]
is independent of $c$  (\cf \eqref{eq:wtccS^1>}) and we set $\wt\cS^1_>=(S^1_\tau\times\CC_t)\sqcup\wt\cS^1_{\infty,>}$.

\begin{lemme}\label{lem:wtcL}
We have $\bR j_{t,*}\scrL_{<0}=j_{t,*}\scrL_{<0}=:\wt\scrL_{<0}$ and the latter sheaf satisfies
\begin{starequation}\label{eq:wtcL}
\wt\scrL_{<0}|_{S^1_\tau\times S^1_t}=(\wt q_\infty^{-1}\cL)_{\wt\cS^1_{\infty,>}},
\end{starequation}%
\ie it is the extension by zero of the restriction of $\wt q_\infty^{-1}\cL$ to the open subset $\wt\cS^1_{\infty,>}$, where~$\wt q_\infty$ denotes the projection $S^1_\tau\times S^1_t\to S^1_\tau$.
\end{lemme}

\begin{proof}
The statement is local, so we can fix a small product neighborhood $U\times \ov V$ of $(\theta,\theta_t)$ in $S^1_\tau\times\wt\PP^1_t$ and choose a decomposition $\cL=\bigoplus_{c\in C}\gr_c\cL$ on $U$. Let us set $V=\ov V\cap\CC_t$. Then, on $U\times V$, we~have $\scrL_{<0}|_{U\times V}\simeq\bigoplus_{c\in C}\gr_c\cL|_U\otimes\bmk_{\cS^1_{c<t}}|_{U\times V}$, and we can argue with each summand independently. We~are thus reduced to showing
\begin{equation}\label{eq:jt}
\bR j_{t,*}(\bmk_{\cS^1_{c<t}}|_{U\times V})\simeq\bmk_{\wt\cS^1_>}|_{U\times\ov V}.
\end{equation}
Choosing $\theta$ in the open interval $U$, we~have $\bmk_{\cS^1_{c<t}}|_{U\times V}\simeq\bmk_U \boxtimes(\bmk_{\cS^1_{c<t,\theta}}|_{V})$. The set $\{t\mid c<_\theta t\}$ is an open half-plane. The closure of the boundary line intersects $S^1_t$ in two points $\theta_{t,1},\theta_{t,2}$. Away from these boundary points, the isomorphism \eqref{eq:jt} holds on $\ov V\cap S^1_t$, and it remains to be shown that the germ of $\bR j_{t,*}\bmk_{\cS^1_{c<t,\theta}}$ at $\theta_{t,1},\theta_{t,2}$ is zero. We~can choose open neighborhoods $\ov{V_i}$ of $\theta_{t,i}$ ($i=1,2$) in $S^1_\tau\times\wt\PP^1_t$ to be homeomorphic to a closed disc with $\theta_{t,i}$ on its boundary, and $\bmk_{\cS^1_{c<t,\theta}}$ is the constant sheaf on an open half-disc of $\ov{V_i}$, extended by zero on the boundary line (a diameter of $\ov{V_i}$). The desired assertion amounts to $H^*(\ov V_i,\bmk_{\cS^1_{c<t,\theta}})=0$, which is the content of Lemma \ref{lem:vanishing}.
\end{proof}

\begin{lemme}\label{lem:qcLneg}
We have $\bR\wt q_*\scrL_{<0}=0$.
\end{lemme}

\begin{proof}
Abusing notation, we~also denote by $\wt q$ the projection $S^1_\tau\times\wt\PP^1_t\to S^1_\tau$. According to Lemma \ref{lem:wtcL}, the assertion amounts to $\bR\wt q_*\wt\scrL_{<0}=0$. Now that $\wt q$ is made proper, we~can argue fiberwise, and we fix $\theta\in S^1_\tau$. We~can thus assume that $(\cL,\cL_\bbullet)$ is graded, so that $\wt\scrL_{<0,\theta}$ can be decomposed with respect to the chosen grading, and we can argue with each $c\in C$ separately.

Let us fix $c_o\!\in\! C$, and assume that $\cL_\theta\!=\!\gr_{c_o}\cL_\theta$. Then, using the notation of~\eqref{eq:wtccS^1}, $\scrL_{<0,\theta}\simeq\gr_{c_o}\!\cL_{\theta}\otimes\bmk_{\wt\cS^1_{c_o<t,\theta}}$ (\ie the constant sheaf with fiber $\gr_{c_o}\!\cL_{\theta}$ on $\wt\cS^1_{c_o<t,\theta}$, extended by zero). The conclusion follows from \eqref{eq:wtccS^1cohom}.
\end{proof}

The lemma yields $\bR\Gamma(\CC_t,\bR\wt p_*\scrL_{<0})=\bR\wt q_*\bR\wt p_*\scrL_{<0}\simeq\bR\wt p_*\bR\wt q_*\scrL_{<0}=0$, ending the proof of Proposition \ref{prop:pushLneg}.
\end{proof}

We end this section by showing how to recover the Stokes-filtered local system $(\cL,\cL_\bbullet)$ on $S^1_\tau$ from the pair $(\wt\scrL=\wt q^{-1}\cL,\wt\scrL_{<0})$ defined on $S^1_\tau\times\wt\PP^1_t$.

\begin{proposition}\label{prop:LneqL}
For each $s\in\CC_t$, we~have
\[
R^k\wt q_*(\wt\scrL_{<0}\otimes\bmk_{\wt\cS^1_{t\geq s}})=0=R^k\wt q_*(\wt q^{-1}\cL\otimes\bmk_{\wt\cS^1_{t\geq s}})\quad\text{for }k\neq0,
\]
and the natural morphism
\[
\wt q_*(\wt\scrL_{<0}\otimes\bmk_{\wt\cS^1_{t\geq s}})\to \wt q_*(\wt q^{-1}\cL\otimes\bmk_{\wt\cS^1_{t\geq s}})\simeq\cL
\]
is an isomorphism onto $\cL_{<s}$.
\end{proposition}

\begin{proof}
On the one hand, we~have a natural isomorphism
\[
\bR\wt q_*(\wt q^{-1}\cL\otimes\bmk_{\wt\cS^1_{t\geq s}})\simeq\cL\otimes\bR\wt q_*\bmk_{\wt\cS^1_{t\geq s}}
\]
and we claim that the natural morphisms $\bR\wt q_*\bmk_{\wt\cS^1_{t\geq s}}\to\bR\wt q_*\wt q^{-1}\bmk_{S^1_\tau}\from \bmk_{S^1_\tau}$ are isomorphisms. Indeed, it is enough to check this fiber by fiber. The assertion follows from the property that the germ $\wt\cS^1_{t\geq s,\theta}$, as~well as $\wt q^{-1}(\theta)$, has the topology of a closed disc.

On the other hand, we~have $\bR\wt q_*(\wt\scrL_{<0}\otimes\bmk_{\wt\cS^1_{t\geq s}})\simeq\bR\wt q_*(\wt\scrL_{<0}\otimes\bmk_{\wt\cS^1_{t<s}}[1])$, because of Lemma \ref{lem:qcLneg}. For $\theta\in S^1_\tau$, we~have $\wt\scrL_{<0,\theta}\simeq\bigoplus_{c\in C}(\gr_c\cL_\theta\otimes\bmk_{\wt\cS^1_{c<t,\theta}})$, so that
\[
\wt\scrL_{<0,\theta}\otimes\bmk_{\wt\cS^1_{t< s},\theta}\simeq\bigoplus_{c\in C}(\gr_c\cL_\theta\otimes\bmk_{\wt\cS^1_{c<t<s,\theta}}).
\]
The term of index $c$ in the sum vanishes iff $c\not<_\theta s$. On the other hand, if $c<_\theta s$, we~have $\bR\Gamma(\wt\PP^1_t,\bmk_{\wt\cS^1_{c<t,\theta}})=0$ by Lemma \ref{lem:vanishing} and an exact sequence
\[
0\to\bmk_{\wt\cS^1_{c<t<s,\theta}}\to\bmk_{\wt\cS^1_{c<t,\theta}}\to\bmk_{\wt\cS^1_{t\geq s}}\to0.
\]
From the above isomorphisms we obtain
\begin{align*}
\bR\Gamma(\wt\PP^1_t,\wt\scrL_{<0,\theta}\otimes\bmk_{\wt\cS^1_{t\geq s,\theta}})
&\simeq\bR\wt q_*(\wt\scrL_{<0}\otimes\bmk_{\wt\cS^1_{t<s}}[1])\\
&\simeq\bigoplus_{c\in C}(\gr_c\cL_\theta\otimes\bR\Gamma(\wt\PP^1_t,\bmk_{\wt\cS^1_{c<t<s,\theta}}[1]))\\
&\simeq\bigoplus_{c<_\theta s}(\gr_c\cL_\theta\otimes\bR\Gamma(\wt\PP^1_t,\bmk_{\wt\cS^1_{t\geq s,\theta}}))\\
&\simeq\bigoplus_{c<_\theta s}\gr_c\cL_\theta=\cL_{<s,\theta}.\qedhere
\end{align*}
\end{proof}

\subsubsection{The local Laplace transformation on \texorpdfstring{$\Perv^0(\CC_t,\bmk,C)$}{Perv}}\label{subsec:FG}
Let us recall the construction of the functor $F_+^{\infty,\top}$ done in \cite[Chap.\,7]{Bibi10}. We~will work with constructible sheaves instead of perverse sheaves, for the sake of simplicity.

Let $\cF$ be a constructible sheaf in $\Perv^0(\CC_t,\bmk,C)[-1]$. Since it is locally constant near $t=\infty$, it extends as a sheaf $\wt\cF=j_{t,*}\cF\simeq\bR j_{t,*}\cF$ on $\wt\PP^1_t$ which is locally constant near $S^1_t$. Recall that $\wt\cS^1_<$ is defined in \eqref{eq:S1<}.

\begin{proposition}[{\cite[Prop.\,7.11]{Bibi10}}]\label{prop:GL}
The complex $\bR\wt q_*(\wt p^{-1}\wt\cF\otimes\bmk_{\wt\cS^1_{t<s}})$ on $S^1_\tau$ has cohomology in degree one at most for each $s\in\CC$, and so has the complex \hbox{$\bR\wt q_*(\wt p^{-1}\wt\cF\otimes\bmk_{\wt\cS^1_<})$}. Moreover, $\cL':=R^1\wt q_*(\wt p^{-1}\wt\cF\otimes\bmk_{\wt\cS^1_<})$ is a locally constant sheaf on~$S^1_\tau$ and $\cL'_{<s}\!:=\!R^1\wt q_*(\wt p^{-1}\wt\cF\otimes\bmk_{\wt\cS^1_{t<s}})$ defines a co-Stokes filtration indexed by~$C$.\qed
\end{proposition}

From the exact sequence \eqref{eq:exactS} and since $\wt\cF$ has vanishing global cohomology, we~deduce
\begin{equation}\label{eq:GL}
\cL'_{<s}\simeq \wt q_*(\wt p^{-1}\wt\cF\otimes\bmk_{\wt\cS^1_{t\geq s}})\quand R^k\wt q_*(\wt p^{-1}\wt\cF\otimes\bmk_{\wt\cS^1_{t\geq s}})=0\text{ for }k\geq1.
\end{equation}

\begin{corollaire}\label{cor:GL}
For $s\in\CC_t$ fixed, we~have $H^k(S^1_\tau,\cL'_{<s})=0$ for $k\neq1$ and a canonical identification $H^1(S^1_\tau,\cL'_{<s})\simeq\cF_s$. In~particular, $\dim\cF_s=\rk\cL'$ if $s\notin C$.
\end{corollaire}

\begin{proof}
The last part follows from Lemma \ref{lem:HkFL}. For the first part, it suffices to show, according to  \eqref{eq:GL}, that $\bR\Gamma(S^1_\tau,\cL'_{<s})\simeq\cF_s[-1]$. We~write
\begin{align*}
\bR\Gamma(S^1_\tau,\cL'_{<s})&\simeq\bR\Gamma(S^1_\tau,\bR\wt q_*(\wt p^{-1}\wt\cF\otimes\bmk_{\wt\cS^1_{t\geq s}}))
\simeq\bR\Gamma(S^1_\tau\times\wt\PP^1_t,(\wt p^{-1}\wt\cF\otimes\bmk_{\wt\cS^1_{t\geq s}}))\\
&\simeq\bR\Gamma(\wt\PP^1_t,\bR\wt p_*(\wt p^{-1}\wt\cF\otimes\bmk_{\wt\cS^1_{t\geq s}}))
\simeq\bR\Gamma(\wt\PP^1_t,\wt\cF\otimes\bR\wt p_*\bmk_{\wt\cS^1_{t\geq s}}).
\end{align*}
As $\bR\Gamma(\wt\PP^1_t,\wt\cF)=0$, the conclusion follows from Lemma \ref{lem:wtccS^1cohom}.
\end{proof}

\subsubsection{Proof of Proposition \ref{prop:Laplacelocsyst}}\label{subsec:proofLaplacelocsyst}
Now that we have at our disposal both functors of~\eqref{eq:Ftop}, we~will prove that they are quasi-inverse from each other.

Given $\cF$ in $\Perv^0(\CC_t,\bmk,C)[-1]$ and extended as $\wt\cF$ to $\PP^1_t$, we~reconstruct the sheaf $\scrL_{<0}$ by varying the parameter~$s$ in Proposition \ref{prop:GL}. For that purpose, we~consider the product $S^1_\tau\times\wt\PP_t\times\CC_s$ and the projections $\wt p_t,\wt p_s,\wt q_s$ to~$\wt\PP_t$, $\wt\PP_s$ and $S^1_\tau\times\CC_s$, respectively. We~consider the subset $\wt\scrS^1_{t<s}$ of $S^1_\tau\times\wt\PP_t\times\CC_s$ whose fiber at $s\in\CC_s$ is the set $\wt\cS^1_{t<s}$. We~also set $\wt\scrS^1_<=\wt\cS^1_<\!\times\CC_s$. Therefore, $\wt\scrS^1_{t<s}\subset\wt\scrS^1_<$. Furthermore, the intersection $\wt\scrS^1_{t<s}\cap(S^1_\tau\times S^1_t\times\CC_s)$ is equal to the product $\wt\cS^1_{\infty,<}\times\CC_s$.

\begin{lemme}
The subset $\wt\scrS^1_{t<s}$ is open in $S^1_\tau\times\wt\PP_t\times\CC_s$.\qed
\end{lemme}

Proposition \ref{prop:GL} in family reads:

\begin{proposition}\label{prop:GLfamily}
The complex $\bR\wt q_{s,*}(\wt p_t^{-1}\wt\cF\otimes\bmk_{\wt\scrS^1_{t<s}})$ on $S^1_\tau\times\CC_s$ has cohomology in degree one at most, as~well as the complex $\bR\wt q_{s,*}(\wt p_t^{-1}\wt\cF\otimes\bmk_{\wt\scrS^1_<})$. Moreover, the sheaf $\scrL'=R^1\wt q_{s,*}(\wt p_t^{-1}\wt\cF\otimes\bmk_{\wt\scrS^1_<})$ is isomorphic to $\wt q_s^{-1}\cL'$ (defined by Proposition \ref{prop:GL}) and the sheaf $\scrL_{\sX,<0}=R^1\wt q_{s,*}(\wt p_t^{-1}\wt\cF\otimes\bmk_{\wt\scrS^1_{t<s}})$ satisfies Lemma \ref{lem:Lneg}. Furthermore, we~have a canonical identification
\[
\bR\wt p_{s,*}\scrL_{\sX,<0}[1]\simeq\cF,
\]
where $\cF$ is now regarded on $\CC_s$.
\end{proposition}

This proposition yields an isomorphism of functors $F^{\infty,\top}_-\circ F^{\infty,\top}_+\isom\id$.

\begin{proof}
The morphisms are defined in a way similar to what is done in Section \ref{subsec:FG}. We~then argue fiber by fiber with respect to the proper projection $S^1_\tau\times\wt\PP^1_t\times\CC_s\to\CC_s$ and apply Proposition \ref{prop:GL} and Corollary \ref{cor:GL}.
\end{proof}

We now construct an isomorphism of functors $\id\isom F^{\infty,\top}_+\circ F^{\infty,\top}_-$. For that purpose, we~start from $(\cL,\cL_\bbullet)$ and \emph{define} $\cF:=\bR\wt p_*\scrL_{<0}[1]$, with $\scrL_{<0}$ as in Section~\ref{subsec:FStokes}. Then~$\cF$ is an object of $\Perv^0(\CC_t,\bmk,C)[-1]$ (Proposition \ref{prop:pushLneg}). With the previous notation, we~then have $\wt\cF=\bR\wt p_*\wt\scrL_{<0}[1]$. We~denote by $(\cL',\cL'_\bbullet)$ the Stokes-filtered local system $F^{\infty,\top}_+(\cF)$.

\begin{proposition}\label{prop:L'L}
There exists a canonical morphism of Stokes-filtered local systems
\[
(\cL,\cL_\bbullet)\to(\cL',\cL'_\bbullet)
\]
such that the induced morphism $\cL'\to\cL$ is an isomorphism of local systems.
\end{proposition}

Since the category of Stokes-filtered local systems is abelian (\cf\eg \cite[Chap.\,3]{Bibi10}), that $\cL\to\cL'$ is an isomorphism implies that the corresponding morphism $(\cL,\cL_\bbullet)\to(\cL',\cL'_\bbullet)$ of Stokes-filtered local systems is an isomorphism. We~conclude that there exists an isomorphism of functors $\id\isom F^{\infty,\top}_+\circ F^{\infty,\top}_-$.

\begin{proof}[Proof of Proposition \ref{prop:L'L}]
We first define, for each $s\!\in\!\CC_t$, a canonical morphism
\begin{equation}\label{eq:LLneg}
\cL_{<s}\to\cL'_{<s}:=\bR\wt q_*(\wt p^{-1}\wt\cF\otimes\bmk_{\wt\cS^1_{t<s}})[1].
\end{equation}
As in \eqref{eq:GL}, we~write the right-hand side as
\[
\cL'_{<s}=\wt q_*(\wt p^{-1}\wt\cF\otimes\bmk_{\wt\cS^1_{t\geq s}})\simeq\bR\wt q_*(\wt p^{-1}\wt\cF\otimes\bmk_{\wt\cS^1_{t\geq s}}),
\]
since we know by (the proof of) Lemma \ref{lem:qcLneg} that $\wt\cF=\bR\wt p_*\wt\scrL_{<0}[1]$ has vanishing global cohomology. On the other hand, since $\wt\scrL_{<0}$ is an $\RR$-constructible sheaf on $S^1_\tau\times\wt\PP^1_t$ and since $\wt p$ is smooth of relative dimension one, we~can apply Poincaré bi-duality to the adjunction morphism $\wt p^{-1}\bR\wt p_*\to\id$ to obtain a canonical morphism
\begin{equation}\label{eq:dualadj}
\wt\scrL_{<0}\to\wt p^{-1}\bR\wt p_*\wt\scrL_{<0}[1]=\wt p^{-1}\wt\cF.
\end{equation}
From Proposition \ref{prop:LneqL} we obtain the desired morphism \eqref{eq:LLneg}:
\begin{equation}\label{eq:LLnegbis}
\wt q_*(\wt\scrL_{<0}\otimes\bmk_{\wt\cS^1_{t\geq s}})\to\wt q_*(\wt p^{-1}\wt\cF\otimes\bmk_{\wt\cS^1_{t\geq s}}).
\end{equation}
This is a morphism of Stokes-filtered local systems, that is, for any $\theta\in S^1_\tau$, if $s<_\theta s'$, the following diagram commutes:
\[
\xymatrix{
\cL_{<s,\theta}\ar[r]\ar@{^{ (}->}[d]&\cL'_{<s,\theta}\ar@{^{ (}->}[d]\\
\cL_{<s',\theta}\ar[r]&\cL'_{<s',\theta}
}
\]
We wish to prove that the morphism $\cL\to\cL'$ that we deduce from these morphisms is an isomorphism. We~know from Lemma \ref{lem:HkFL} and Corollary \ref{cor:GL} that~$\cL$ and~$\cL'$ have the same rank, so it is enough to prove that it is onto. We~can also fix $\theta\in S^1_\tau$, and we note that, for $s\in\CC_t$ such that $\wt\cS^1_{t\geq s,\theta}$ is a closed half-plane not inter\-sect\-ing~$C$, we~have $\cL_{<s,\theta}=\cL_\theta$, and $\cL'_{<s,\theta}=\cL'_\theta$. We~then fix such an~$s$. The fiber at $\theta$ of the morphism \eqref{eq:LLnegbis} reads
\[
\Gamma(\wt\PP^1_t,\scrL_{<0,\theta}\otimes\bmk_{\wt\cS^1_{t\geq s,\theta}})\to\Gamma(\wt\PP^1_t,\wt\cF\otimes\bmk_{\wt\cS^1_{t\geq s,\theta}})
\]
and since the sheaves $\scrL_{<0,\theta}$ and $\wt\cF$ are locally constant on the closed disc $\wt\cS^1_{t\geq s,\theta}$, the above morphism is nothing but the morphism between their fibers at $s$, which reads
\begin{equation}\label{eq:LH1L}
\cL_{<s,\theta}\to H^1(S^1_\tau,\cL_{<s})
\end{equation}
and is the fiber at $(\theta,s)$ of \eqref{eq:dualadj}. We~are thus left with proving that the latter is onto. By~construction, this morphism is the fiber at $\theta$ of the morphism
\[
\cL_{<s}\to H^1(S^1_\tau,\cL_{<s})\boxtimes\CC_{S^1_\tau}
\]
which is dual to the adjunction morphism
\[
\bH^0(S^1_\tau,(\cL_{<s})^\vee)\boxtimes\CC_{S^1_\tau}\to(\cL_{<s})^\vee,
\]
where $\cI^\vee$ denotes the Poincaré dual complex $\rHom(\cI,\CC_{S^1_\tau})$. On the other hand, we~know (\cf\cite[Lem.\,2.16]{Bibi10}) that the Poincaré dual complex $\rHom(\cL_{<s},\CC_{S^1_\tau})$ is a sheaf (\ie concentrated in degree zero), so that the adjunction morphism
\[
H^0(S^1_\tau,(\cL_{<s})^\vee)\boxtimes\CC_{S^1_\tau}\to(\cL_{<s})^\vee
\]
is injective. By~duality, we~obtain that \eqref{eq:LH1L} is onto, as~desired.
\end{proof}

\subsubsection{Proof of the equivalence \texorpdfstring{\eqref{eq:equivPervSt}}{equiv}}\label{subsec:equivPervSt}
In a first step, we~work with the real blow-up~$\wh\PP^1_\tau$ of~$\PP^1_\tau$ at~$\infty$ and $0$, so that $\wh\PP^1_\tau$ is homeomorphic to a cylinder $S^1_\tau\times[0,\infty]$. A~Stokes-constructible sheaf on $\wh\PP^1_\tau$ consists of a local system $\wh\cL$ on $\wh\PP^1_\tau$ together with a Stokes filtration on the local systems $\cL=\whi_\infty^{-1}\wh\cL$ and $\whi_0^{-1}\wh\cL$. In~our setting, we~consider a Stokes filtration $\cL_\bbullet$ of exponential type $C$ on $\cL$ and a trivial Stokes filtration on $\whi_0^{-1}\wh\cL$. Furthermore, due to Lemma \ref{lem:recol}, there exists a unique filtration~$\wh\cL_\bbullet$ of~$\wh\cL$ such that the restriction of $\wh\cL_\bbullet$ away from $S^1_\tau\times\{\infty\}$ is equal to the restriction of $\wh\cL$ (\ie is the trivial filtration), while its restriction to $S^1_\tau\times\{\infty\}$ is equal to $\cL_\bbullet$.

Since $\wh\cL$ is the pullback of $\cL$ by the projection $\wh\PP^1_\tau\to S^1_\tau$, Lemma \ref{lem:recol} yields an equivalence
\begin{equation}\label{eq:wtLL}
(\wh\cL,\wh\cL_\bbullet)\To{\whi_\infty^{-1}}(\cL,\cL_\bbullet).
\end{equation}

\begin{lemme}
The shifted Stokes-constructible sheaf $(\wh\cL,\wh\cL_\bbullet)[1]$ is a Stokes-perverse sheaf on $\wh\PP^1_\tau$, in the sense of \cite[Def.\,4.12]{Bibi10}.
\end{lemme}

\begin{proof}
It is a matter of proving that, \emph{for any $s\in\CC$, the Poincaré-Verdier dual complexes $\bD(\wt\cL_{<s}[1])$ and $\bD(\wh\cL_{\leq s}[1])$ have nonzero cohomology in degree one only}. This being clear away from $S^1_\tau\times\{\infty\}$, where both are equal to the shifted local system $\wh\cL[1]$, the question is local at a point $(\theta,\infty)$ of $S^1_\tau\times\{\infty\}$. We~can the assume that $(\cL,\cL_\bbullet)$ is graded, as~well as $(\wh\cL,\wh\cL_\bbullet)$. We~consider an open neighborhood of the form $I\times[0,\epsilon)$ for some small interval $I$ of $S^1_\tau$ containing $\theta$. We~set $V=I\times(0,\epsilon)$, $\ov V= I\times[0,\epsilon)$ and, decomposing $I$ as $I_1\sqcup I_2\sqcup\{\theta\}$, where $\theta$ is the common boundary point of the open intervals $I_1,I_2$, we~set $\ov V_i=V\sqcup (I_i\times\{0\})$, so that $V\subset \ov V_i\subset\ov V$ ($i=1,2$). We~are led to consider three cases for either $\wh\cL_{<s}[1]$ or $\wh\cL_{\leq s}[1]$, depending on whether $\theta$ is a Stokes direction for $s$ or not:
\begin{enumerate}
\item
the shifted sheaf $\bmk_V[1]$ (\ie its extension by zero to $I\times\{0\}$),
\item
the shifted sheaf $\bmk_{\ov V}[1]$,
\item
the shifted sheaf $\bmk_{\ov V_i}[1]$ ($i=1,2$).
\end{enumerate}
Then one checks that the Poincaré-Verdier dual of $\bmk_V[1]$ is $\bmk_{\ov V}[1]$, and that of $\bmk_{\ov V_1}[1]$ is $\bmk_{\ov V_2}[1]$ (and conversely). This yields the assertion, which proves the lemma.
\end{proof}

Let us now consider the real blow-up map $\varpi_0:\wh\PP^1_\tau\to\wt\PP^1_\tau$ of $\wt\PP^1_\tau$ at $\tau=0$. Then arguing as in \cite[Prop.\,4.22]{Bibi10}, one checks that $(\wt\cG,\wt\cG_\bbullet):=\bR\varpi_{0,*}(\wh\cL,\wh\cL_\bbullet)$ is an object of $\PervSt(\wt\PP^1_\tau,\bmk,C,*0)$ and furthermore the correspondence $(\wh\cL,\wh\cL_\bbullet)\mto(\cG,\cG_\bbullet)$ is an equivalence. Composing it with the equivalence \eqref{eq:wtLL} yields the equivalence \eqref{eq:equivPervSt}. In~particular, $\cG=\wt\cG|_{\CC_\tau}$, satisfies $\cG\simeq\bR j_{0*}j_0^{-1}\cG$, and $j_0^{-1}\cG[-1]$ is a locally constant sheaf on $\CC^*_\tau$, that can be identified with $\cL$.

We also note the following:

\begin{lemme}\label{lem:FL}
For each $t\in\CC$ and $k\in\NN$, the following equalities hold:
\[
H^k(\wt\PP^1_\tau,\wt\cG_{<t})=H^k(S^1_\tau,\cL_{<t})\quand H^k(\wt\PP^1_\tau,\wt\cG_{\leq t})=H^k(S^1_\tau,\cL_{\leq t}).
\]
\end{lemme}

\begin{proof}
We argue for $\cG_{\leq t}$ for example. We~have a long exact sequence
\[
\cdots\to H^k_\rc(\CC_\tau,\cG)\to H^k(\wt\PP^1_\tau,\wt\cG_{\leq t})\to H^k(S^1_\tau,\cL_{\leq t})\to\cdots
\]
and $H^k_\rc(\CC_\tau,\cG)=0$ for any $k$, according to Lemma \ref{lem:vanishing}.
\end{proof}

\section{The Riemann-Hilbert correspondence}\label{sec:RH}

Let $\cM$ be a free $\CC\{\tau'\}[\tau^{\prime -1}]$-module of finite rank with a connection $\nabla$ having an irregular singularity of exponential type, with exponential factors $c/\tau'$, with $c$ in a finite subset $C\subset\CC$. In~other words, $\cM$ is an object of $\mathsf{Mero}(\infty,C)$. The construction of Deligne's meromorphic extension together with an algebraization procedure for meromorphic connections on $\PP^1_\tau$ yields a unique free $\CC[\tau,\tau^{-1}]$-module with connection~$M$ such that, setting $\tau'=\tau^{-1}$, we have $\CC\{\tau'\}\otimes_{\CC[\tau']}M\simeq\cM$ and the connection on $M$ has a regular singularity at $\tau=0$ and no other singularity in the complex plane $\CC_\tau$. The correspondence $\cM\mto M$ is a quasi-inverse functor to the functor $i_\infty^{-1}$ considered in the diagram \eqref{eq:main**}.

On the other hand, in the diagram \eqref{eq:main**}, we~\emph{define}
\[
F^\infty_\pm:\Mod_{\holreg}^0(\cD_{\Afu_t},C)\to\mathsf{Mero}(\infty,C)
\]
so that $F^\infty_\pm=i_\infty^{-1}\circ F_\pm$, and
\[
F^\infty_\mp:\mathsf{Mero}(\infty,C)\to\Mod_{\holreg}^0(\cD_{\Afu_t},C)
\]
so that $F_\mp\circ i_\infty^{-1}=F^\infty_\mp$. In~\cite[Chap.\,7]{Bibi10}, it is proved that the diagram
\[
\xymatrix@C=1.5cm{
\Mod_{\holreg}^0(\cD_{\Afu_t},C)\ar[d]_{\mathrm{RH}}\ar[r]^-{F^\infty_\pm}&\mathsf{Mero}(\infty,C)\ar[d]^-{\mathrm{RHB}}\\
\Perv^0(\CC_t,\CC,C)[-1]\ar[r]^-{F^{\infty,\top}_\pm}&\Sto(S^1_\tau,\CC,C)
}
\]
commutes. In~this section, we~prove that, with Definition \ref{def:Finftop} of $F^{\infty,\top}_\mp$, the diagram
\[
\xymatrix@C=1.5cm{
\Mod_{\holreg}^0(\cD_{\Afu_t},C)\ar[d]_{\mathrm{RH}}&\ar[l]_-{F^\infty_\mp}\mathsf{Mero}(\infty,C)\ar[d]^-{\mathrm{RHB}}\\
\Perv^0(\CC_t,\CC,C)[-1]&\ar[l]_-{F^{\infty,\top}_\mp}\Sto(S^1_\tau,\CC,C)
}
\]
also commutes. We~will focus on $F^\infty_-,F^{\infty,\top}_-$. Together with the equivalences $\wti_\infty^{-1}$ seen in Section \ref{subsec:equivPervSt} and $i_\infty^{-1}$ seen above, the proof of Theorem \ref{th:main} is completed.

Given $\cM$ in $\mathsf{Mero}(\infty,C)$, that we express as $i_\infty^{-1}M$, the~Fourier transform~$N=F_-M$ is defined as $\D p_*(\D q^*M,\nabla-\rd(\tau t))$. We~set
\[
\scrM:=(\CC\{\tau'\}[t]\otimes_{\CC\{\tau'\}}\cM,\nabla-\rd(\tau t)).
\]
This is the analytic germ along $\{\infty\}\times\CC_t$ of $(\D q^*M,\nabla-\rd(\tau t))$, the latter being regarded as a holonomic $\cD_{\PP^1_\tau\times\CC_t}$-module.

Let $(\cL,\cL_\bbullet)$ be the $\CC$-Stokes-filtered local system associated to $\cM$ on $S^1_\tau$ by the Riemann-Hilbert-Birkhoff-Deligne-Malgrange (RHB) correspondence and let $\scrL_{<0}$ be the sheaf on $S^1_\tau\times\CC_t$ defined in Section \ref{subsec:FStokes}. We~will simply denote by $\varpi$---instead of $\varpi\times\id$---the real blowing up (\cf Diagram \eqref{eq:diag1}) $\wt\PP^1_\tau\times\CC_t\to\PP^1_\tau\times\CC_t$ as well as its restriction $\wt D_\infty=S^1_\tau\times\CC_t\to D_\infty=\{\infty\}\times\CC_t$ above $D_\infty$.

\begin{theoreme}\label{th:compFou}
There exists a functorial isomorphism $\DR^\an(\scrM)\simeq\bR\varpi_*\scrL_{<0}$.
\end{theoreme}

Set $\scrG=\DR^\an(\D q^*M,\nabla-\rd(\tau t))$, which a constructible complex on $\PP^1_\tau\times\CC_t$. Letting $j_\infty:\CC_\tau\times\CC_t\hto\PP^1_\tau\times\CC_t$ and $i_\infty:\{\infty\}\times\CC_t\hto\PP^1_\tau\times\CC_t$ denote the complementary inclusions, we~have $\DR^\an(\scrM)\simeq i_\infty^{-1}\scrG$ and $\bR p_*\bR j_{\infty,!}\,j_\infty^{-1}\scrG=0$ (by~Lemma \ref{lem:vanishing}), so~that
\[
\pDR^\an(N)\simeq\bR p_*\pDR^\an(\D q^*M,\nabla-\rd(\tau t))=\bR p_*\scrG[2]\simeq\bR p_*\DR^\an(\scrM)[2].
\]
According to the commutation of the analytic de~Rham functor with proper pushforward of holonomic $\cD$-modules, we~obtain, according to Theorem \ref{th:compFou}:

\begin{corollaire}\label{cor:RHB}
The analytic de~Rham complex $\pDR^\an(N)$ is isomorphic to $\bR\wt p_*\scrL_{<0}[2]$.\qed
\end{corollaire}

\begin{remarque}
Since $\pDR^\an(N)$ is known to be a perverse complex, and has zero global cohomology (\cf the introduction), we~recover the properties of Proposition \ref{prop:pushLneg} for $\bR\wt p_*\scrL_{<0}[2]$ when $(\cL,\cL_\bbullet)$ is defined over any subfield of $\CC$.
\end{remarque}

\subsubsection{First reduction for the proof of Theorem \ref{th:compFou}}
We realize the isomorphism stated in Theorem \ref{th:compFou} as the pushforward by $\varpi$ of an isomorphism on $\wt\PP^1_\tau\times\CC_t$. For that purpose, let us consider the sheaf $\cA_{\wt\PP^1_\tau\times\CC_t}^\rmod$ on $\wt\PP^1_\tau\times\CC_t$ consisting of holomorphic functions on $\CC_\tau\times\CC_t$ having moderate growth along $\wt D_\infty$, and let $\cA_{S^1_\tau\times\CC_t}^\rmod$ denote its sheaf-theoretic restriction to $S^1_\tau\times\CC_t$. We~similarly set $\cO_{\wt\PP^1_\tau\times\CC_t}=\varpi^{-1}\cO_{\PP^1_\tau\times\CC_t}$, $\cO_{S^1_\tau\times\CC_t}=\cO_{\wt\PP^1_\tau\times\CC_t}|_{S^1_\tau\times\CC_t}$ and $\Omega^{k,\an}_{\{\infty\}\times\CC_t}=\Omega^{k,\an}_{\PP^1_\tau\times\CC_t}|_{\{\infty\}\times\CC_t}$ ($k\geq0$). We~then define $\DR^\rmod(\scrM)$ as the complex
\[
\Bigl(\cA_{S^1_\tau\times\CC_t}^\rmod \otimes_{\varpi^{-1}\cO_{S^1_\tau\times\CC_t}}\varpi^{-1}(\scrM\otimes\Omega^{\cbbullet,\an}_{\{\infty\}\times\CC_t}),\nabla-\rd(\tau t)\Bigr).
\]

The theorem is then a consequence of:

\begin{proposition}\label{prop:RHB}
There exists a functorial isomorphism $\DR^\rmod(\scrM)\simeq\scrL_{<0}$.
\end{proposition}

\begin{proof}[Proof of Theorem \ref{th:compFou}]
We use Proposition \ref{prop:RHB} and the isomorphism
\[
\bR\varpi_*\DR^\rmod(\scrM)\simeq\DR^\an(\scrM),
\]
which follows from the isomorphism $\bR\varpi_*\cA_{S^1_\tau\times\CC_t}^\rmod\simeq\cO_{\{\infty\}\times\CC_t}(*D_\infty)$.
\end{proof}

The remaining part of this section is devoted to the proof of Proposition \ref{prop:RHB}.

\subsubsection{The complex and real blowing up of \texorpdfstring{$\PP^1_\tau\times\CC_t$}{PC}}\label{subsec:RH}

The proof of Proposition \ref{prop:RHB} cannot be easily done on $\PP^1_\tau\times\CC_t$ as we do not have the suitable asymptotic analysis at our disposal, due to the non-goodness of $\Phi\cup\{0\}$. The solution to this problem in the present setting consists in blowing up $\PP^1_\tau\times\CC_t$ at each point $(\infty,c)\in D_\infty$. This is a birational morphism $e:X\to\PP^1_\tau\times\CC_t$ which is an isomorphism away from the points $(\infty,c)$ with $c\in C$. In~particular, it is an isomorphism above $\CC_\tau\times\CC_t$. The pullback $D=e^{-1}(D_\infty)$ is a normal crossing divisor, which is the union
\[
D=D'_\infty\cup E,\quad E=\bigsqcup_{c\in C}E_c,
\]
where $D'_\infty$ is the strict transform of $D_\infty$ (so that $e:D'_\infty\to D_\infty$ is an isomorphism) and $E$ is the exceptional divisor, with each $E_c=e^{-1}(c)$ isomorphic to $\PP^1$.

Correspondingly, we~consider the real blowing up $\varpi_\sX:\wt X\to X$ of the irreducible components of $D$. The diagram \eqref{eq:diag1} can be completed as the diagram~\eqref{eq:diag2}.
\begin{equation}\label{eq:diag2}
\begin{array}{c}
\xymatrix@R=.5cm@C=1.2cm{
S^1_\tau\ar@{^{ (}->}[d]\ar@/_2.1pc/[dddd]&\wt D_\infty
\ar[l]_-{\wt q}
\ar@{_{ (}->}[d]
\ar@/^2.05pc/[dddd]|(.25)\hole|(.37)\hole|(.67)\hole|(.75)\hole|(.77)\hole
&&\wt D\ar[ll]_-{\wt e}\ar@/^2.1pc/[dddd]\ar@{_{ (}->}[d]
\\
\wt\PP^1_\tau\ar[dd]_-{\varpi}&\wt\PP^1_\tau\times\CC_t\ar[l]_-{\wt q}\ar[dd]_-{\varpi\times\id}\ar[rd]^-{\wt p}
&&\wt X\ar[ll]_-{\wt e}\ar[ld]_-{\wt\pi}\ar[dd]^{\varpi_\sX}\\
&&\CC_t\\
\PP^1_\tau&\PP^1_\tau\times\CC_t\ar[l]_-{q}\ar[ru]_-{p}
&&X\ar[ll]_-{e}\ar[lu]^-{\pi}\\
\{\infty\}\ar@{_{ (}->}[u]&D_\infty\ar@{^{ (}->}[u]\ar[l]_-{q}
&&D\ar@{^{ (}->}[u]\ar[ll]_-{e}
}
\end{array}
\end{equation}

We fix $c_o\in C$ and we denote by $\Delta_x=\Delta_x(c_o)\subset\CC_t$ the open disc centered at~$c_o$, with coordinate $x=t-c_o$, of radius $r>0$ small enough such that $\Delta_x$ contains no other point of~$C$. We~restrict the diagram \eqref{eq:diag2} to $\PP^1_\tau\times\Delta_x$, and we still denote by~$X$ the open subset $e^{-1}(\PP^1_\tau\times\Delta_x)$, and similarly with $\wt X$. Then $X$ is covered by the following charts.

\begin{enumeratea}
\item\label{enum:charta}
The chart with coordinates $(u,v)$ such that $e(u,v)=(\tau'=uv,x=u)$; in this chart, the divisor $D$ is defined by $uv=0$, the strict transform $D'_\infty$ of $D_\infty$ is defined by $v=0$ and has coordinate $u$, while the exceptional divisor $E$ is defined by $u=0$ and has coordinate~$v$; its pullback in $\wt X$ has polar coordinates $((\rho_u,\theta_u),(\rho_v,\theta_v))$ with
\[
\rho_{\tau'}=\rho_u\rho_v,\quad\rho_x=\rho_u,\quad\theta=\theta_u+\theta_v,\quad\theta_x=\theta_u;
\]
this chart is the open set of $(\RR_+)^2\times(S^1)^2$ defined by $\rho_u\in[0,r)$; the pullback~$\wt D$ of~$D$ is identified with $\partial(\RR_+)^2\times(S^1)^2$, \ie the subset where $\rho_u$ or $\rho_v$ is zero; we~regard it as the (disjoint) union of $\wt D'_\infty=\{\rho_v=0\}$ and $S^1_u\times\CC^*_v$; of special interest is $\varpi_\sX^{-1}(0,0)=S^1_u\times S^1_v$.

\item\label{enum:chartb}
The chart with coordinates $(z,w)$ such that $e(z,w)=(\tau'=z,x=wz)$; in this chart, the divisor $D$ coincides with the exceptional divisor $E$ and is defined as $z=0$, so that $w$ is a coordinate on it; its pullback in $\wt X$ has polar coordinates $((\rho_z,\theta_z),w)$ with
\[
\rho_{\tau'}=\rho_z,\quad\theta=\theta_z;
\]
the pullback $\wt E$ is identified with $S^1_z\times\CC_w$ and its projection to $\wt\PP^1\times\Delta_x$ is nothing but the projection of $S^1_z\times\CC_w$ to $S^1_\tau\times\{0\}$ defined by $(\theta_z,w)\mto(\theta=\theta_z,0)$.
\end{enumeratea}

Let us make explicit the relations between the two charts on their intersection $\CC^*_v\times\Delta_u$, defined respectively by $v\neq0$ and $w\neq0$. It~is also identified with the open subset $\{z\neq0,\,|w|<r/|z|\}$ of $\CC^*_w\times\CC_z$. We~have $w=1/v$ (this is the change of chart on $E\simeq\PP^1$) and $z=uv$. Concerning polar coordinates, we~have $\rho_v\neq0$ and $w\neq0$ on the intersection, and $\theta_z=\theta_u+\theta_v$.

\subsubsection{The set of exponential factors \texorpdfstring{$\Phi$}{Phi}}
We consider the set $\Phi=\{\varphi_c\mid c\in C\}$ of exponential factors of $(\scrL,\scrL_\bbullet)$ and its pullback $\Phi_\sX=\{\psi_c:=\varphi_c\circ e\mid c\in C\}$. In~the local setting of Section \ref{subsec:RH}, $\Phi_\sX$ is as follows:
\begin{itemize}
\item
in the chart $(u,v)$, $\psi_c=(c-c_o)/uv-1/v$,
\item
in the chart $(z,w)$, $\psi_c=(c-c_o)/z$ modulo holomorphic functions.
\end{itemize}
Then $\Phi_\sX$ is a finite set of global sections of $\cO_\sX(*D)/\cO_\sX$, that we regard as defining a stratified covering $\Sigma\to D$ with $\Sigma$ contained in the sheaf space of $\cO_\sX(*D)/\cO_\sX$ (\cf\cite[Chap.\,1]{Bibi10} for this notion). The stratified property comes from the fact that the pole divisor of $\psi_{c_o}$ is equal to the strict transform $D'_\infty$ of $D_\infty$, and not to $D'_\infty\cup E_{c_o}$. Then~$\Phi_\sX$ is good at each point of $D$ since $\Phi$ as good at each point of~$D_\infty$. Even better, $\Phi_\sX$ is \emph{very good} since for $c\neq c'\in C$, $\psi_c-\psi_{c'}$ has a pole along \emph{each} component of $D$.

\begin{lemme}
Let $\wt x\in \wt D$. For $c\neq c'\in C$, we~have
\begin{align*}
\psi_c\leq_{\wt x}\psi_{c'}\iff \varphi_c\leq_{\wt e(\wt x)}\varphi_{c'},\\
\tag*{and}
\psi_c<_{\wt x}\psi_{c'}\iff \varphi_c<_{\wt e(\wt x)}\varphi_{c'}.
\end{align*}
\end{lemme}

\begin{proof}
The second equivalence follows from the first one, since for $c\neq c'$, $\psi_c-\psi_{c'}$ is nowhere holomorphic on $D$:
\begin{itemize}
\item
in the chart $(u,v)$, $\psi_c-\psi_c'=(c-c')/uv$,
\item
in the chart $(z,w)$, $\psi_c-\psi_c'=(c-c')/z$ modulo holomorphic functions.\qedhere
\end{itemize}
\end{proof}

\subsubsection{The Stokes-filtered local system \texorpdfstring{$(\scrL_\sX,\scrL_{\sX,\bbullet})$ on $\wt D$}{LLDX}}

Let $\scrL_\sX$ be the local system $\wt e^{-1}\scrL=(\wt e\circ\wt q)^{-1}\cL$ and let us set $\scrL_{\sX,\leq\psi_c}=\wt e^{-1}\scrL_{\leq\varphi_c}=(\wt e\circ\wt q)^{-1}\cL_{\leq c}$, and similarly with $<\psi_c$. Then $\gr_{\psi_c}\scrL_\sX=(\wt e\circ\wt q)^{-1}\gr_c\!\cL$ is a local system and, according to the lemma above, the filtration condition for $\scrL_{\sX,\bbullet}$ holds at~$\wt x$ iff it holds at $\wt e(\wt x)$ for $\scrL_\bbullet$, \ie iff it holds at $(\wt e\circ\wt q)(\wt x)\in S^1_\tau$ for $\cL_\bbullet$. It~follows that $(\scrL_\sX,\scrL_{\sX,\bbullet})$ is a Stokes-filtered local system on $\wt D$ indexed by $\Phi_\sX$.

We now consider the set $\Phi_\sX\cup\{0\}$. We~note that it is good at each point of $D$: for $c\in C$, $\psi_c-0=\psi_c$ is purely monomial both if $c\neq c_o$ and if $c=c_o$, where we recall that $c_o$ has been fixed at the beginning of the proof. (However, it is not very good any more, since $\psi_{c_o}$ does not have a pole along $E_{c_o}$.) We~can then extend $(\scrL_\sX,\scrL_{\sX,\bbullet})$ as a Stokes-filtered local system indexed by $\Phi_\sX\cup\{0\}$ by setting, for each $\wt x\in\wt D$:
\[
\scrL_{\sX,\leq0,\wt x}=\sum_{c\mid\psi_c\leq_{\wt x}0}\scrL_{\sX,\leq\psi_c,\wt x},\quad \scrL_{\sX,<0,\wt x}=\sum_{c\mid\psi_c<_{\wt x}0}\scrL_{\sX,\leq \psi_c,\wt x}.
\]
As the subsets $\{\wt x\mid\psi_c\leq_{\wt x}\psi_{c'}\}$ and $\{\wt x\mid\psi_c\leq_{\wt x}0\}$ are open, $\scrL_{\sX,\leq0}$ is a subsheaf of~$\scrL_\sX$. On the other hand, the subsets $\{\wt x\mid\psi_c<_{\wt x}0\}$ are not open, as~for example in the local setting of Section \ref{subsec:RH}, for $\wt x\in S^1_u\times S^1_v$, we~have $\psi_{c_o}\leq_{\wt x}0$ iff $\theta_v\in(-\pi/2,\pi/2)$ (with $\wt x=(\theta_u,\theta_v)$), and this is equivalent to $\psi_{c_o}<_{\wt x}0$, but the latter property does not hold at any nearby point $\wt y$ above a point of the exceptional divisor $E_{c_o}\moins\{(\infty,c_o)\}$. We~thus find:
\begin{itemize}
\item
on $\varpi_\sX^{-1}(D'_\infty)$, $\scrL_{\sX,<0}$ is a subsheaf of $\scrL_\sX$ which coincides with $\scrL_{\sX,\leq0}$,
\item
on $\varpi_\sX^{-1}(E_{c_o}\moins\{(\infty,c_o)\})$ where $\psi_{c_o}$ has no pole, we~have $\scrL_{\sX,\leq0}=(\wt q\circ\wt e)^{-1}\cL_{\leq c_o}$ and $\scrL_{\sX,<0}=(\wt q\circ\wt e)^{-1}\cL_{< c_o}$.
\end{itemize}
This illustrates the ``stratified'' definition of a Stokes-filtered local system along a normal crossing divisor (\cf\cite[\S1.9]{Bibi10}).

\subsubsection{Proof of Proposition \ref{prop:RHB}}

We consider the sheaf $\cA_{\wt X}^\rmod$ on $\wt X$ of holomorphic functions on $\CC_\tau\times\CC_t$ having moderate growth along $\wt D$ and its sheaf-theoretical restriction $\cA_{\wt D}^\rmod$ to $\wt D$. We~denote by $\scrM_\sX$ the pullback meromorphic bundle with flat connection $\D e^*\scrM$. We~can thus define the moderate de~Rham complex $\DR_{\wt D}^\rmod(\scrM_\sX)$, which is a complex on $\wt D$. The proof of Proposition \ref{prop:RHB} is done in three steps.

\begin{enumerate}
\item
The Hukuhara theorem in dimension two yields:

\begin{theoreme}\label{th:hukuhara}
We have $\cH^k(\DR_{\wt D}^\rmod(\scrM_\sX))=0$ if $k\neq0$ and a natural isomorphism $\cH^0\DR_{\wt D}^\rmod(\scrM_\sX)\simeq\scrL_{\sX,\leq0}$.
\end{theoreme}

\item
The moderate de~Rham complex behaves well by blowing down:

\begin{proposition}\label{prop:wtL'L}
There exists a natural isomorphism
\[
\bR\wt e_*\,\DR_{\wt D}^\rmod(\scrM_\sX)\simeq\DR_{S_\tau^1\times\CC_t}^\rmod(\scrM).
\]
\end{proposition}

\begin{proof}
See \eg\cite[Prop.\,8.9]{Bibi10}.
\end{proof}

\item
Similarly, the pushforward of $\scrL_{\sX,\leq0}$ by $\wt e$ is isomorphic to $\scrL_{<0}$:

\begin{proposition}\label{prop:F'F}
We have $R^k\wt e_*\scrL_{\sX,\leq0}=0$ for $k\neq0$ and the natural morphism $\wt e_*\scrL_{\sX,\leq0}\to\wt q^{-1}\cL$ induces an isomorphism $\wt e_*\scrL_{\sX,\leq0}\simeq\scrL_{<0}$, and so, $\bR\wt e_*\scrL_{\sX,\leq0}\simeq\scrL_{<0}$.
\end{proposition}
\end{enumerate}

We obtain Proposition \ref{prop:RHB} by applying $\bR\wt e_*$ to the isomorphism of Theorem \ref{th:hukuhara} and then Propositions \ref{prop:wtL'L} and \ref{prop:F'F} on each side.\qed

\subsubsection{Proof of Theorem \ref{th:hukuhara}}
For the first part, \cf\eg\cite[Cor.\,12.7]{Bibi10}. For the second part, we~first exhibit a morphism $\cH^0\DR_{\wt D}^\rmod(\scrM)\to\scrL_\sX$ and then show that it is an isomorphism onto $\scrL_{\sX,\leq0}$. For that purpose, let us set $\cL^*=j_\infty^{-1}\DR^\an(M)$, which is a local system  on~$\CC_\tau^*$. Then, $j_\infty^{-1}\DR^\an(\D q^*M,\D q^*\nabla-\rd(\tau t))$ is a local system  on $\CC^*_\tau\times\CC_t$, canonically isomorphic to $\scrL^*:=q^{-1}\cL^*$.

Let us denote now by $\wtj_\infty,\wti_\infty$ the pair of complementary inclusions $\CC^*_\tau\times\CC_t\hto\wt X$ and $\wt D\hto X$.

\begin{lemme}\label{lem:L'L*}
We have $\scrL_\sX=\wti_\infty^{-1}\wtj_{\infty,*}\scrL^*\simeq\wti_\infty^{-1}\bR\wtj_{\infty,*}\scrL^*$.
\end{lemme}

\begin{proof}
The analysis of neighborhoods made in Section \ref{subsec:RH}, \eqref{enum:charta} and \eqref{enum:chartb} shows that there exists a basis of open neighborhoods $U$ of a point of $S^1_\tau\times\CC_t$ in $\wt\PP^1_\tau\times\CC_t$, or~of a point of $\wt D$ in $\wt X$ which is not above a crossing point of $D$, of the form $I\times[0,\epsilon)\times\Delta$, where $I$ is an open interval in $S^1$ and $\Delta$ is an open disc in $\CC$. Above a crossing point of $D$, the neighborhoods take the form $I^2\times[0,\epsilon)^2$. The trace of $\CC^*_\tau\times\CC_t$, or~$X\moins D$, on such an open neighborhood is then $I\times(0,\epsilon)\times\Delta$ or $I^2\times(0,\epsilon)^2$. In~any case, this is a contractible open set.  It follows that
\[
i_\infty^{-1}\bR j_{\infty,*}\scrL^*\simeq i_\infty^{-1}j_{\infty,*}\scrL^*\quand \wti_\infty^{-1}\bR\wtj_{\infty,*}\scrL^*\simeq\wti_\infty^{-1}\wtj_{\infty,*}\scrL^*.
\]
On the other hand, we have a natural morphism $\wt e^{-1}j_{\infty,*}\scrL^*\to\wtj_{\infty,*}\scrL^*$, and thus, by applying $\wti_\infty^{-1}$, a morphism $\scrL_\sX=\wt e^{-1}\scrL\to\wti_\infty^{-1}\wtj_{\infty,*}\scrL^*$. That it is an isomorphism is thus a local question at points of $\wt E:=\varpi_\sX^{-1}(\bigcup_cE_c)$. As we have seen above, any $\wt x\in\wt E$ has a basis of open neighborhoods $\wt U$ in $\wt X$ such that $\wt U^*=\wt U\moins\wt D$ is contractible. A similar property holds for its image $\wt e(\wt x)=(\theta_{\tau'},c)\in\wt D_\infty=S^1_\tau\times\CC_t$ with $\wt V\subset \wt\PP^1_\tau\times\CC_t$ and $\wt V^*=\wt V\moins(S^1_\tau\times\CC_t)$. On the one hand, $\scrL_{\sX,\wt x}=\varinjlim_{\wt V^*}\Gamma(\wt V^*,\scrL^*)$, and $(\wti_\infty^{-1}\wtj_{\infty,*}\scrL^*)_{\wt x}=\varinjlim_{\wt U^*}\Gamma(\wt U^*,\scrL^*)$. If $\wt e(\wt U^*)\subset\wt V^*$, then the restriction morphism $\Gamma(\wt V^*,\scrL^*)\to\Gamma(\wt U^*,\scrL^*)$ is an isomorphism since~$\scrL^*$ is locally constant and both $\wt U^*$ and $\wt V^*$ are contractible. This concludes the proof.
\end{proof}

According to this lemma, from the natural morphism
\[
\cH^0\DR_{\wt D}^\rmod(\scrM_\sX)\to \wtj_{\infty,*}\wtj_\infty^{-1}\cH^0\DR_{\wt D}^\rmod(\scrM_\sX)=\wtj_{\infty,*}\scrL^*
\]
we obtain, by applying $\wti_\infty^{-1}$, the desired morphism
\[
\cH^0\DR_{\wt D}^\rmod(\scrM_\sX)\to\wti_\infty^{-1}\wtj_{\infty,*}\scrL^*=\scrL_\sX.
\]

We now show that the image of the latter morphism is $\scrL_{\sX,\leq0}$. This is clear away from the pullback of the points $(\infty,c)$. Since the question is local, we fix~$\wt x$ projecting to $(\infty,c_o)\in D_\infty$ with $c_o\in C$. Up to tensoring all data, we can, and will, assume that $c_o=0$, so that the coordinate $x=t-c_o$ is equal to $t$ (\cf Section~\ref{subsec:RH}).

In the chart $(z,w)$, we have $e^*\rd(\tau t)=e^*\rd(t/\tau')=\rd w$, so that $\scrM_X$ is locally isomorphic to $\D e^*\D q^*\cM$, and $\psi_c=c/z=c/\tau'$, so that the Stokes-filtered local system attached to $\scrM_\sX$, which is indexed by $\Phi_X$, is the pullback by $\wt q\circ\wt e$ of that of~$M$, which is indexed by $C/\tau'$.

In the chart $(u,v)$, we have $e^*\rd(\tau t)=e^*\rd(x\tau)=\rd(1/v)$, and $\psi_c=c/\tau'-\nobreak1/v$ if $c\neq\nobreak0$ and $\psi_0=-1/v$. Hence $\scrM_X$ is obtained from $\D e^*\D q^*\cM$ by twisting the connection by $-\rd(1/v)$ and $\Phi_X$ is obtained from $\Phi$ by adding $-1/v$ to $\Phi$. We~conclude that the Stokes-filtered local system of $\scrM_X$ on $S^1_u\times S^1_v$ is simply the pullback of that of~$\cM$ by the map $(\theta_u,\theta_v)\mto\theta_u+\theta_v$, and this concludes the proof of Theorem~\ref{th:hukuhara}.
\qed

\subsubsection{Proof of Proposition \ref{prop:F'F}}
Away from $S^1_\tau\times C\subset S^1_\tau\times\CC_t$, the map $\wt e:\wt D\to S^1_\tau\times\CC_t$ is an isomorphism, and the assertion amounts to $\scrL_{\leq0}=\scrL_{<0}$, which follows from the definition of $\scrL_{<0}$ in Section \ref{subsec:FStokes}.

Let us fix $c_o\!\in\! C$. The inverse image $\wt e^{-1}(S^1_\tau\times\{c_o\})$ is isomorphic to $\wt E_{c_o}\times S^1_z$, where $\wt E_{c_o}$ is the real blow-up of the exceptional divisor $E_{c_o}$ at $w=\infty$, so that $\wt E_{c_o}$ is diffeomorphic to a closed disc, with boundary $S^1_w$. The map
\[
\wt e:\wt E_{c_o}\times S^1_z\to S^1_\tau\times\{c_o\}
\]
is induced by the projection $\wt E_{c_o}\times S^1_z\to S^1_z$ composed with the isomorphism \hbox{$S^1_z\isom S^1_\tau$} induced by $\theta_z\mto\theta$. We~will describe the restriction to $\wt e^{-1}(S^1_\tau\times\{c_o\})$ of the sheaves~$\scrL_{\sX,\leq0}$ (and we now omit to mention this restriction functor).

\begin{itemize}
\item
Let us set $E_{c_o}^\circ\!=\!E_{c_o}\moins\infty\!=\!\wt E_{c_o}\moins S^1_w$. On $E_{c_o}^\circ\!\times\! S^1_z$, $\scrL_{\sX,\leq0}$ coincides with $\wt e^{-1}(\cL_{\leq c_o})$.
\item
On $S^1_w\times S^1_z$, we~have inclusions
\[
\wt e^{-1}(\cL_{< c_o})\subset \scrL_{\sX,\leq0}\subset\wt e^{-1}(\cL_{\leq c_o}).
\]
More precisely, locally with respect to $S^1_z$ that we identify with $S^1_\tau$, we~have decompositions
\[
\cL_{\leq c_o,\theta_z}\simeq\bigoplus_{c\leq_{\theta_z}c_o}\gr_c\!\cL_{\theta_z}\quand \cL_{< c_o,\theta_z}\simeq\bigoplus_{\substack{c\neq c_o\\c\leq_{\theta_z}c_o}}\gr_c\!\cL_{\theta_z}.
\]
Then at a point $(\theta_w,\theta_z)\in S^1_w\times S^1_z$, we~have
\[
\scrL_{\sX,\leq0,(\theta_w,\theta_z)}=
\begin{cases}
\bigoplus_{c\leq_{\theta_z}c_o}\gr_c\!\cL_{\theta_z}&\text{if }\theta_w\in(-\pi/2,\pi/2),\\
\cL_{< c_o,\theta_z}\simeq\bigoplus_{\substack{c\neq c_o\\c\leq_{\theta_z}c_o}}\gr_c\!\cL_{\theta_z}&\text{otherwise}.
\end{cases}
\]
\end{itemize}
In conclusion, let $U_{c_o}$ be the open subset of $\wt E_{c_o}$ which is the complement of the closed subset~$[\pi/2,3\pi/2]$ of $S^1_w$. Then, on $\wt E_{c_o}\times S^1_z$, we~have an exact sequence
\[
0\to \wt e^{-1}(\cL_{< c_o})\to \scrL_{\sX,\leq0}\to \scrQ\to0,
\]
where the quotient sheaf $\scrQ$ is isomorphic to the sheaf
\[
(\wt e^{-1}\gr_{c_o}\!\cL)_{U_{c_o}\times S^1_z}\simeq \CC_{U_{c_o}}\boxtimes \gr_{c_o}\!\cL
\]
(\ie the extension by zero of its restriction to $U_{c_o}\times S^1_z$). We~have
\begin{align*}
\bR\wt e_*\wt e^{-1}(\cL_{< c_o})&\simeq \bR\wt e_*\CC_{\wt E_{c_o}}\otimes \cL_{< c_o}\simeq \cL_{< c_o}\quad \text{($\wt E_{c_o}$ is a closed disc),}\\
\bR\wt e_*(\CC_{U_{c_o}}\boxtimes \gr_{c_o}\!\cL)&\simeq\bR\Gamma_\rc(U_{c_o})\otimes\gr_{c_o}\!\cL=0\quad(\text{Lemma \ref{lem:vanishing}}).
\end{align*}
We conclude that the natural morphism $\scrL_{<0}|_{S^1_\tau\times\{c_o\}}\simeq\cL_{<c_o}\hto \wt e_*\scrL_{\sX,\leq0}\simeq \bR \wt e_*\scrL_{\sX,\leq0}$ is an isomorphism.\qed

\backmatter
\bibliographystyle{amsalpha}
\bibliography{sabbah_three-results}

\providecommand{\eprint}[1]{\href{http://arxiv.org/abs/#1}{\texttt{arXiv\string:\allowbreak#1}}}
\providecommand{\bysame}{\leavevmode\hbox to3em{\hrulefill}\thinspace}
\providecommand{\MR}{\relax\ifhmode\unskip\space\fi MR }
\providecommand{\MRhref}[2]{%
  \href{http://www.ams.org/mathscinet-getitem?mr=#1}{#2}
}
\providecommand{\href}[2]{#2}
\begin{thebibliography}{Moc11b}

\bibitem[BG12]{B-G12}
A.~Beilinson and D.~Gaitsgory, \emph{A corollary of the b-function lemma},
  Selecta Math. (N.S.) \textbf{18} (2012), no.~2, 319--327.

\bibitem[DK16]{D-K13}
A.~D'Agnolo and M.~Kashiwara, \emph{Riemann-{H}ilbert correspondence for
  holonomic {D}-modules}, Publ. Math. Inst. Hautes {\'E}tudes Sci. \textbf{123}
  (2016), 69--197.

\bibitem[HTT08]{H-T-T08}
R.~Hotta, K.~Takeuchi, and T.~Tanisaki, \emph{{$D$-Modules, perverse sheaves,
  and representation theory}}, Progress in Math., vol. 236, Birkh{\"a}user,
  Boston, Basel, Berlin, 2008, in Japanese: 1995.

\bibitem[Ked10]{Kedlaya09}
K.~Kedlaya, \emph{Good formal structures for flat meromorphic connections,~{I}:
  surfaces}, Duke Math.~J. \textbf{154} (2010), no.~2, 343--418.

\bibitem[Ked11]{Kedlaya10}
\bysame, \emph{Good formal structures for flat meromorphic connections,~{II}:
  excellent schemes}, J.~Amer. Math. Soc. \textbf{24} (2011), no.~1, 183--229.

\bibitem[KKP08]{K-K-P08}
L.~Katzarkov, M.~Kontsevich, and T.~Pantev, \emph{{Hodge theoretic aspects of
  mirror symmetry}}, {From Hodge theory to integrability and TQFT:
  tt*-geometry} (R.~Donagi and K.~Wendland, eds.), Proc. Symposia in Pure
  Math., vol.~78, American Mathematical Society, Providence, RI, 2008,
  pp.~87--174.

\bibitem[Mac74]{MacPherson74}
{\relax R.D}.~MacPherson, \emph{Chern classes for singular varieties}, Ann. of
  Math. \textbf{100} (1974), 423--432.

\bibitem[Mai23]{Maisonobe16}
{\relax Ph}.~Maisonobe, \emph{Filtration relative, l'id\'eal de {B}ernstein et
  ses pentes}, Rend. Sem. Mat. Univ. Padova \textbf{150} (2023), 81--125.

\bibitem[Mal91]{Malgrange91}
B.~Malgrange, \emph{{\'E}quations diff{\'e}rentielles {\`a} coefficients
  polynomiaux}, Progress in Math., vol.~96, Birkh{\"a}user, Basel, Boston,
  1991.

\bibitem[Mal96]{Malgrange95}
\bysame, \emph{{Connexions m{\'e}romorphes, II: le r{\'e}seau canonique}},
  Invent. Math. \textbf{124} (1996), 367--387.

\bibitem[Meb90]{Mebkhout90}
Z.~Mebkhout, \emph{Le th{\'e}orème de positivit{\'e} de l'irr{\'e}gularit{\'e}
  pour les {$\mathcal{D}_{X}$}-modules}, {The Grothendieck Festschrift},
  Progress in Math., vol. 88, no.~3, Birkh{\"a}user, Basel, Boston, 1990,
  pp.~83--132.

\bibitem[Meb04]{Mebkhout04}
\bysame, \emph{{Le th{\'e}orème de positivit{\'e}, le th{\'e}orème de
  comparaison et le th{\'e}orème d'existence de Riemann}}, {{\'E}l{\'e}ments
  de la th{\'e}orie des systèmes diff{\'e}rentiels g{\'e}om{\'e}triques},
  S{\'e}minaires \& Congrès, vol.~8, Soci{\'e}t{\'e} Math{\'e}matique de
  France, Paris, 2004, pp.~165--310.

\bibitem[Moc09a]{Mochizuki07b}
T.~Mochizuki, \emph{{Good formal structure for meromorphic flat connections on
  smooth projective surfaces}}, {Algebraic Analysis and Around (Kyoto, June
  2007)}, Advanced Studies in Pure Math., vol.~54, Math. Soc. Japan, Tokyo,
  2009, pp.~223--253.

\bibitem[Moc09b]{Mochizuki09}
\bysame, \emph{{On Deligne-Malgrange lattices, resolution of turning points and
  harmonic bundles}}, Ann. Inst. Fourier (Grenoble) \textbf{59} (2009), no.~7,
  2819--2837.

\bibitem[Moc11a]{Mochizuki10b}
\bysame, \emph{{Stokes structure of a good meromorphic flat bundle}}, J.~Inst.
  Math. Jussieu \textbf{10} (2011), no.~3, 675--712.

\bibitem[Moc11b]{Mochizuki08}
\bysame, \emph{{Wild harmonic bundles and wild pure twistor $D$-modules}},
  Ast{\'e}\-risque, vol. 340, Soci{\'e}t{\'e} Math{\'e}matique de France,
  Paris, 2011.

\bibitem[Moc14]{Mochizuki10}
\bysame, \emph{{Holonomic $\mathcal D$-modules with Betti structure}}, M{\'e}m.
  Soc. Math. France (N.S.), vol. 138--139, Soci{\'e}t{\'e} Math{\'e}matique de
  France, Paris, 2014.

\bibitem[Moc15]{Mochizuki11}
\bysame, \emph{{Mixed twistor D-modules}}, Lect. Notes in Math., vol. 2125,
  Springer, Heidelberg, New York, 2015.

\bibitem[Moc18]{Mochizuki18}
\bysame, \emph{{Stokes shells and Fourier transforms}}, 2018,
  \eprint{1808.01037}.

\bibitem[MS89]{M-S86b}
Z.~Mebkhout and C.~Sabbah, \emph{{\S\kern .15em III.4 {$\mathcal{D}$}-modules
  et cycles {\'e}vanescents}}, {Le formalisme des six op{\'e}rations de
  Grothendieck pour les {$\mathcal{D}$}\nobreakdash-mo\-du\-les coh{\'e}rents},
  Travaux en cours, vol.~35, Hermann, Paris, 1989, pp.~201--239.

\bibitem[Rib24]{Ribeiro24}
G.~Ribeiro, \emph{{Generic vanishing for holonomic D-modules: a study via
  Cartier duality}}, Ph.D. thesis, Institut polytechnique de Paris, 2024.

\bibitem[Sab93]{Bibi93}
C.~Sabbah, \emph{{{\'E}quations diff{\'e}rentielles {\`a} points singuliers
  irr{\'e}guliers en dimension~$2$}}, Ann. Inst. Fourier (Grenoble) \textbf{43}
  (1993), 1619--1688.

\bibitem[Sab00]{Bibi97}
\bysame, \emph{{{\'E}quations diff{\'e}rentielles {\`a} points singuliers
  irr{\'e}guliers et ph{\'e}\-no\-mène de Stokes en dimension~{$2$}}},
  Ast{\'e}\-risque, vol. 263, Soci{\'e}t{\'e} Math{\'e}matique de France,
  Paris, 2000.

\bibitem[Sab13]{Bibi10}
\bysame, \emph{{Introduction to Stokes structures}}, Lect. Notes in Math., vol.
  2060, Springer-Verlag, 2013.

\bibitem[Sab17]{Bibi16b}
\bysame, \emph{{A remark on the irregularity complex}}, J.~Singul. \textbf{16}
  (2017), 100--114, Erratum: \emph{Ibid.}, p.\,194.

\bibitem[Tey23]{Teyssier17}
J.-B. Teyssier, \emph{{Moduli of Stokes torsors and singularities of
  differential equations}}, J.~Eur. Math. Soc. (JEMS) \textbf{25} (2023),
  369--411.

\bibitem[Wu21]{Wu21}
L.~Wu, \emph{{Riemann-Hilbert correspondence for Alexander complexes }},
  \eprint{2104.06941}, 2021.

\bibitem[YZ24]{Y-Z24}
T.~Y. Yu and S.~Zhang, \emph{{Topological Laplace transform and decomposition
  of nc-Hodge structures}}, 2024, \eprint{2405.19549}.

\end{thebibliography}
\end{document}